\documentclass{amsart}
\usepackage{amsthm,amsmath,verbatim,amssymb}
\newtheorem{lem}{Lemma}
\newtheorem{defn}[lem]{Definition}
\newtheorem{prop}[lem]{Proposition}
\newtheorem{thm}[lem]{Theorem}

\newtheorem{cor}[lem]{Corollary}

\theoremstyle{remark}
\newtheorem{rem}[lem]{Remark}

\newcommand{\ddbar}{i\partial\overline{\partial}}
\newcommand{\dbar}{\overline{\partial}}
\renewcommand{\phi}{\varphi}
\renewcommand{\epsilon}{\varepsilon}
\renewcommand{\l}{\mathbf{l}}
\newcommand{\Bl}{\mathrm{Bl}}

\title{Blowing up extremal K\"ahler manifolds II}
\author{G\'abor Sz\'ekelyhidi}
\address{Department of Mathematics, University of Notre Dame, Notre
  Dame, IN 46615}
\email{gszekely@nd.edu}
\date{}

\begin{document}

\begin{abstract}
	This is a continuation of the work of Arezzo-Pacard-Singer and
        the author on
	blowups of extremal K\"ahler manifolds. We 
	prove the conjecture stated in~\cite{GSz10}, and we relate this
	result to the K-stability of blown up manifolds. As an
	application we prove that if a K\"ahler manifold $M$ of
	dimension greater than 2 admits a cscK metric, then the blowup of $M$
	at a point admits
	a cscK metric if and only if it is K-stable, as
	long as the exceptional divisor is sufficiently small.
\end{abstract}
\maketitle

\section{Introduction}
We continue our study~\cite{GSz10} of extremal metrics on blown-up
manifolds, following the work of Arezzo-Pacard~\cite{AP06,AP09} and
Arezzo-Pacard-Singer~\cite{APS06}. See Pacard~\cite{Pac10} for a
survey and see also LeBrun-Singer~\cite{LS93},
Rollin-Singer~\cite{RS07}, Tipler~\cite{Tip11},
Biquard-Rollin~\cite{BR12}
for related work. The starting point is a compact
K\"ahler manifold $M$  with an extremal metric $\omega$.
The notion of extremal metric was introduced by Calabi~\cite{Cal82}, and
it means that the gradient of the scalar curvature
$\nabla\mathbf{s}(\omega)$ is a holomorphic vector field. The basic
question that we study is whether the 
blowup $\Bl_{p_1,\ldots,p_n}M$ of $M$ in a finite number of points admits
an extremal metric in the K\"ahler class
\begin{equation}\label{eq:kclass}
	\pi^*[\omega] - \epsilon^2(a_1[E_1] + a_2[E_2] +\ldots a_n[E_n]),
\end{equation}
where $\pi$ is the blowdown map, $a_1,\ldots,a_n > 0$ are constants, 
$\epsilon > 0$ is very small, and $E_i$ are the exceptional divisors. 
Our methods,
following~\cite{AP06,AP09,APS06} are perturbative, restricting the
results to sufficiently small $\epsilon > 0$. In addition our results
will be restricted to blowing up only one point, and dimension $m >
2$. We expect that with some more work our method can deal
with the case $m=2$, but blowing up more than one point introduces
more serious difficulties as we will
explain in Section~\ref{sec:general}. 

To state the main result, let us write $G$ for the group of Hamiltonian
isometries of $(M,\omega)$ and $\mathfrak{g}$ for its Lie algebra. Let
\begin{equation}
	\mu : M \to \mathfrak{g}^*
\end{equation}
be the equivariant moment map for the action of 
$G$ normalized in such a way that
the Hamiltonian functions $\langle\mu,\xi\rangle$ have zero mean on $M$
for all $\xi\in\mathfrak{g}$. From now on we will identify
$\mathfrak{g}$ with its dual, using the inner product given by the $L^2$
product on Hamiltonian functions. Let $\Delta\mu$ be the Laplacian of $\mu$
taken componentwise after identifying $\mathfrak{g}^*$ with
$\mathbf{R}^l$ for some $l$. A central role is played by the perturbed
moment map 
\[ \mu(p) + \delta\Delta\mu(p) \]
for small $\delta$. 
Note that this is simply the moment map for the action of $G$ on
$M$ with respect to the K\"ahler form $\omega-\delta\rho$, where $\rho$
is the Ricci form of $\omega$. Let us write $G^c$ for the
complexification of $G$, acting on $M$ by biholomorphisms. 
With this the main result is as follows, confirming 
Conjecture 6 in~\cite{GSz10} in the case when $m > 2$. 

\begin{thm}\label{thm:main}
	Assume that the dimension $m > 2$, and suppose that
        $\nabla\mathbf{s}(\omega)$ vanishes at $p\in M$. There is a $\delta_0 > 0$
        depending on $(M,\omega)$ with the following property. Suppose
        that for some $\delta\in(0,\delta_0)$ there is a point $q$ in
        the $G^c$-orbit of $p$ such that the vector field $\mu(q) +
        \delta \Delta\mu(q)$ vanishes at $q$. Then 
	      the blowup $\Bl_pM$ admits an
			extremal metric in the K\"ahler class
			\[ \pi^*[\omega] - \epsilon^2[E],\]
			for all sufficiently small $\epsilon > 0$.
\end{thm}

Suppose that $M$ is a projective variety and $\omega\in c_1(L)$ for a line bundle $L$
over $M$. The condition in the theorem can be interpreted as relative
stability of the point $p$ with respect to the natural linearization
of the $G^c$-action on the $\mathbf{Q}$-line bundle $L + \delta K_M$
for small rational $\delta$, where $K_M$ is the canonical bundle. In
this terminology, our earlier result in~\cite{GSz10} only dealt with
the case when $p$ is relatively stable with respect to the
linearization on $L$. This in turn refined earlier results of
Arezzo-Pacard-Singer~\cite{APS06}, where some extra conditions were
required. Allowing a small perturbation of the line bundle
$L$ gives more precise information about case when the point $p$ is
strictly semistable. 

In the case when $(M,\omega)$ is a constant scalar curvature K\"ahler
(or cscK) manifold, then we can show that
Theorem~\ref{thm:main} actually gives a
complete characterization of the possible blowup points. Suppose again
that $M$ is projective and $\omega\in c_1(L)$. 
For small rational $\epsilon$ let us write
$L_\epsilon = \pi^*L - \epsilon^2[E]$
for an ample $\mathbf{Q}$-line bundle on the blowup
$\Bl_pM$. 
The Yau-Tian-Donaldson conjecture~\cite{Yau93, Tian97,
Don02} predicts that the existence of a cscK metric on the
blowup $\Bl_pM$ in the first Chern class $c_1(L_\epsilon)$
is related to the K-stability of the pair
$(\Bl_pM, L_\epsilon)$. In Section~\ref{sec:relstab}
we define a simple version of K-stability for
K\"ahler manifolds, restricting attention to test-configurations with
smooth central fibers. Using this and Theorem~\ref{thm:main} we obtain
the following. 
\begin{thm}\label{thm:Kstab}
        Let $(M,\omega)$ be a cscK manifold of dimension $m > 2$, and
        let $p\in M$. Then for sufficiently small
        $\epsilon > 0$ the following are equivalent, and are
        independent of $\epsilon$:
        \begin{enumerate}
          \item The blowup $\Bl_pM$ admits a cscK metric
            in the class $\pi^*[\omega] - \epsilon^2[E]$,
          \item The pair $(\Bl_pM,\pi^*[\omega]-\epsilon^2[E])$ is
            K-stable with respect to smooth test-configurations,
          \item There is a point $q$ in the $G^c$-orbit of $p$ such
            that $\mu(q) + \epsilon\Delta\mu(q) =0$.
       \end{enumerate}
\end{thm}

This extends our earlier result in~\cite{GSz10}, where $(M,\omega)$ was
assumed to be a K\"ahler-Einstein manifold. A natural problem is to
generalize this result to extremal metrics, and we will discuss the
difficulty that arises in Section~\ref{sec:general}.  

The contents of the paper are as follows. In
Section~\ref{sec:finitedim} we prove a finite dimensional
perturbation result, which is sharper than the result
we used in~\cite{GSz10}. The heart of the paper is
Section~\ref{sec:gluing} where the main analytic gluing theorem,
Theorem~\ref{thm:gluing}  is
proved, and the proof of Theorem~\ref{thm:main} is given
at the end of
Section~\ref{sec:extequation}.
The main ingredient is a refined expansion of the solution to the equation
introduced in \cite{GSz10}. This is similar to other obstructed
perturbation problems in the literature such as Pacard-Xu~\cite{PX09}, but an
important difference is that computing one more term in our case
is enough to get a sharp existence result for cscK metrics when $n >
2$. The reason for this is
the algebro-geometric structure of the problem. We discuss this
aspect in the final
Section~\ref{sec:relstab}, where we extend \cite[Theorem 5]{GSz10}
to K\"ahler manifolds which are not necessarily
algebraic, and we give the proof of Theorem~\ref{thm:Kstab}. 

\subsection*{Acknowledgements}
I would like to thank Frank Pacard and Michael Singer for several useful discussions.

\section{Relative stability}\label{sec:finitedim}

In this section we will consider the action of a compact Lie group $G$
on a compact K\"ahler manifold $M$, which is Hamiltonian with respect to a
K\"ahler form $\omega$. The action of $G$ extends to a holomorphic
action of the complexification $G^c$. Let us write 
\[ \mu : M \to \mathfrak{g} \]
for the moment map, where we have identified $\mathfrak{g}\cong
\mathfrak{g}^*$ using an invariant inner product. Since we want to
work on K\"ahler manifolds which may not be algebraic, we will review
the basic ideas from relative stability~\cite{GSz04} in this
setting. The usual GIT theory has been extended to this setting by
several authors, for instance Mundet i Riera~\cite{MR00}, and
Teleman~\cite{Tel04}. Our definitions will not necessarily match with
theirs since we just want to cover the bare minimum of the theory that
we will use. 

We will need to work with maximal compact subgroups of $G^c$ other
than $G$. If $K\subset G^c$ is maximal compact, then $K =
\mathrm{Ad}_g(G)$ for some $g\in G^c$. The metric $(g^{-1})^*\omega$
is $K$-invariant, and a corresponding moment map is given by
\begin{equation}\label{eq:muK}
  \mu_K(g\cdot p) = \mathrm{ad}_g \mu(p).
  \end{equation}
This way of assigning ``compatible'' moment maps for the actions of
all maximal compact subgroups of $G^c$ is analogous to a linearization
of the action in GIT, and it is called a ``symplectization'' of the
action in \cite{Tel04}. 

\begin{defn}\label{defn:relstab}
  \begin{enumerate}
    \item A point $p\in M$ is stable for the action of $G^c$ (and for our
choice of symplectization), if the stabilizer $G^c_p$ is trivial, and there
exists a point $q\in G^c\cdot p$ such that $\mu(q) =0$.
\item A point $p\in M$ is semistable for the action of $G^c$ if there
  is a point in the orbit closure $q\in \overline{G^c\cdot p}$ such
  that $\mu(q) =0 $. 
\item A point $p\in M$ is relatively stable for the action of $G^c$,
  if there exists a point $q\in G^c\cdot p$ such that $\mu(q)\in
  \mathfrak{g}_q$. 
  \end{enumerate}
\end{defn}

It follows from \eqref{eq:muK} that
the definition is independent of the choice of maximal compact
subgroup. The main observation in \cite{GSz04} (see also
Kirwan~\cite{Kir84}) is that relative stability of a point $p$ is
equivalent to stability of $p$ for the action of a subgroup of
$G^c$. The relevant subgroup can be defined for any complex torus
$T^c\subset G^c$. For this, let
$K\subset G^c$ be a maximal compact subgroup
containing the compact torus $T$. Let us write
\[ \mathfrak{k}_{T^\perp} = \{ \xi\in \mathfrak{k}\,:\,
\mathrm{ad}_\xi(\eta) = 0, \quad \langle \xi,\eta\rangle = 0\,\text{
  for all }\eta\in\mathfrak{t}\}, \]
and let $K_{T^\perp}$ be the corresponding subgroup of $K$ (this is a
closed subgroup by \cite[Lemma 1.3.2]{GSzThesis}). We then let
$G^c_{T^\perp}$ be the complexification of $K_{T^\perp}$. We will need
the following result. 

\begin{prop}\label{prop:GT}
Suppose that $p\in M$ is relatively stable, and $T\subset G_p$ is a
maximal torus, such that $T^c\subset G^c_p$ is also maximal. Then
there exists a point $q\in G^c_{T^\perp}\cdot p$ such that $\mu(q)\in
\mathfrak{g}_q$.  
\end{prop}
\begin{proof}
  Using \eqref{eq:muK}, we can choose a maximal compact subgroup
  $K\subset G^c$ such that $\mu_K(p)\in \mathfrak{k}_p$. Let
  $T'\subset K_p$ be a maximal torus. It follows then that $p$ has
  trivial stabilizer in the group $K_{T'^\perp}$, and 
  \begin{equation}\label{eq:projk}
    \mathrm{proj}_{\mathfrak{k}_{T'^\perp}} \mu_K(p) = 0,
  \end{equation}
  so from \cite[Lemma 16]{GSz10} we find that the stabilizer of $p$ in
  $G^c_{T'^\perp}$ is trivial. This implies that $T'^c$ is a maximal
  torus in $G^c_p$. We can therefore conjugate $T'^c$ into $T^c$ using
  an element of $G^c_p$, and by taking the corresponding conjugate of
  the maximal compact $K$, we can assume that $T'=T$.

  From \eqref{eq:projk} we then have that $p$ is stable for the action
  of $G^c_{T^\perp}$, and in particular using the maximal compact
  subgroup $G_{T^\perp}$ we find that there is a point $q\in
  G^c_{T^\perp}\cdot p$ such that
  \begin{equation}\label{eq:muq}
    \mathrm{proj}_{\mathfrak{g}_{T^\perp}} \mu(q) = 0.
  \end{equation}
  Since elements in $G^c_{T^\perp}$ commute with $T$, we have that
  $T\subset G_q$. Since $\mu(q)$ is $G_q$ invariant, we have that
  $\mu(q)$ commutes with $\mathfrak{t}$. It then follows from
  \eqref{eq:muq} that $\mu(q)\in\mathfrak{t}$, and in particular
  $\mu(q)\in\mathfrak{g}_q$. This is what we set out to prove. 
\end{proof}

We will also need a K\"ahler version of the Hilbert-Mumford criterion,
developed in \cite{MR00} and \cite{Tel04}. 
For any
$\xi\in\mathfrak{g}$ and $p\in M$ we define the weight
\begin{equation}\label{eq:weight}
 W(p,\xi) = \lim_{t\to -\infty} \langle\mu(e^{it\xi}\cdot p),
\xi\rangle. 
\end{equation}
We then have the following result from Teleman~\cite{Tel04}.
\begin{prop}\label{prop:HM}
  Suppose that the stabilizer $G^c_p$ is trivial. 
\begin{enumerate}
\item The point $p\in M$ is stable, i.e. 
there is $q\in G^c\cdot p$ with $\mu(q)=0$, if and only
  if there is a constant $\delta > 0$, such that $W(p,\xi) > \delta$
  for all $\xi$ with $\Vert \xi\Vert =1$. 
\item The point $p\in M$ is semistable, i.e. there is $q\in
  \overline{G^c\cdot p}$ with $\mu(q)=0$, if and only if $W(p,\xi)
  \geqslant 0$ for all $\xi\in\mathfrak{g}$. 
\end{enumerate}
\end{prop}
\noindent
Note that the strict lower bound $\delta > 0$ in (1) is obtained
because $W(p,\xi)$ as a function on the unit sphere in $\mathfrak{g}$
is lower semicontinuous, being the supremum of a family of continuous
functions.

\subsection{A perturbation problem}
We will now assume that we have two K\"ahler forms $\omega, \eta$ on
$M$, and the action of $G$ is
assumed to be Hamiltonian with respect to both K\"ahler forms. We
have equivariant moment maps
\[ \mu, \nu : M \to \mathfrak{g}. \]
We will therefore have to specify which moment map (or
symplectization) we use when we speak of stable or semistable points. 
In our application we will have $\nu = \mu + \delta_0\Delta\mu$ for
small $\delta_0$, corresponding to the K\"ahler form $\omega -
\delta_0\rho$.  
We will again assume that $p\in M$ has trivial stabilizer. 
Our goal is the following result. 
\begin{thm}\label{thm:perturb}
 Let $p$ have trivial stabilizer. 
Suppose that $\mu_\epsilon : M\to\mathfrak{g}$ is a family of continuous
functions such that
\[ \mu_\epsilon = \mu + \epsilon\nu + O(\epsilon^\kappa), \]
for some $\kappa > 1$. Suppose that for all sufficiently small 
$\epsilon > 0$ we can find  $q\in
G^c\cdot p$ such that $\mu(q) + \epsilon\nu(q)=0$. 
Then for sufficiently small $\epsilon > 0$ we can
find $q\in G^c\cdot p$ such that $\mu_\epsilon(q) = 0$. 
\end{thm}

\begin{proof}
The assumption implies that for 
each small $\epsilon > 0$ we have a point $q_\epsilon \in G^c\cdot p$
such that
\[ \mu(q_\epsilon) + \epsilon\nu(q_\epsilon) = 0. \]
By compactness of $M$, we can assume up to choosing a subsequence,
that $q_\epsilon \to q_0$ as $\epsilon\to 0$, for some $q_0\in
M$. It follows that  
\[ \mu(q_0) = 0. \]
Since $q_0\in \overline{G^c\cdot p}$, the point $p$ is semistable with
respect to $\mu$. If $p$ were actually stable with respect to $\mu$,
then we could apply Proposition 8 from \cite{GSz10} to find $q\in
G^c\cdot p$ for sufficiently small $\epsilon > 0$ such that
$\mu_\epsilon(q)=0$. The difficulty now is that the point $q_0$ may be
on the boundary of the $G^c$-orbit of $p$.

By the assumption, we can choose a small $\delta > 0$ such that $p$ is
stable with respect to $\mu + \delta\nu$. For $\epsilon \ll \delta$
we have
\[ \frac{\delta}{\delta - \epsilon}\mu_\epsilon = \mu +
\frac{\epsilon}{\delta-\epsilon}(\mu + \delta\nu) +
O(\epsilon^\kappa), \]
and so by replacing $\nu$ by $\mu + \delta\nu$, we can assume that $p$
is stable with respect to $\nu$. 

Let us define the weights $W_\mu(p,\xi)$ and $W_\nu(p,\xi)$ as in
\eqref{eq:weight}. We then know from Proposition~\ref{prop:HM} that
$W_\mu(p,\xi) \geqslant 0$ for all $\xi$, and there is a $c_0 > 0$
such that
\[ W_\nu(p,\xi) > c_0, \text{ for all }\xi\text{ with }\Vert\xi\Vert=1. \]
Fix an $\epsilon > 0$. By linearity we have
$W_{\mu+\epsilon\nu}(p,\xi) > 0$ for all $\xi\not=0$, and in fact
\[ W_{\mu+\epsilon\nu}(p,\xi) > \epsilon c_0,\text{ for all $\xi$ with
  $\Vert\xi\Vert=1$.} \]
Using the compactness of the unit sphere, it follows that there
is a large radius $R > 0$ such that
\begin{equation}\label{eq:normmu}
 \Vert \mu(e^{i\xi}\cdot p) + \epsilon\nu(e^{i\xi}\cdot p)
\Vert > \frac{1}{2}\epsilon c_0,\text{
  for all $\xi$ with $\Vert\xi\Vert = R$}. 
\end{equation}
Since $p$ is stable with respect to $\mu + \epsilon\nu$, 
it follows that there is a unique $\xi_\epsilon$ such
that 
\[ \mu(e^{i\xi_\epsilon}\cdot p) + \epsilon \nu(e^{i\xi_\epsilon}\cdot
p) = 0. \]
Consider the maps $F, F_\epsilon : \mathfrak{g} \to
\mathfrak{g}$ given by 
\[ \begin{aligned}
 F(\xi) &= \mu(e^{i\xi}\cdot p) + \epsilon \nu(e^{i\xi}\cdot p)\\
 F_\epsilon(\xi) &= \mu_\epsilon(e^{i\xi}\cdot p).
\end{aligned}\]
From \eqref{eq:normmu} we know that $F$ induces a map 
\[ F : \partial B(R) \to \mathfrak{g}\setminus \{0\}, \]
where $\partial B(R)$ is the $R$-sphere in $\mathfrak{g}$. By a
homotopy argument we can see that the degree of this map is $\pm 1$, since
$F$ also induces a map
\[ F : \mathfrak{g}\setminus \{\xi_\epsilon\} \to
\mathfrak{g}\setminus\{0\}, \]
while at $\xi_\epsilon$ the derivative of $F$ is an isomorphism. If
$\epsilon$ is chosen sufficiently small, then by our assumption 
\[ | F_\epsilon - F | < \frac{1}{4}\epsilon c_0, \]
so $F_\epsilon$ also defines a map with nonzero degree from $\partial
B(R)$ to $\mathfrak{g}\setminus\{0\}$. But then there must be a
$\xi\in\mathfrak{g}$ such that $F_\epsilon(\xi) =0 $. 
\end{proof}

We also need the following, which is analogous to Proposition 12 in
\cite{GSz10}, but applies in the non-algebraic case as well. 
\begin{prop}\label{prop:alldelta}
Assume again that the stabilizer of $p$ is trivial. 
There is a $\delta_0 > 0$, depending on $\mu, \nu$, such that the
following are equivalent:
\begin{itemize}
\item[(1)] For some $\delta\in(0,\delta_0)$ we can find $q\in G^c\cdot
  p$ such that $\mu(q) + \delta\nu(q)=0$.
\item[(2)] We have $W_\mu(p,\xi) \geqslant 0$ for all $\xi$, and
  $W_\nu(p,\xi) > 0$ for all $\xi\not=0$ for which $W_\mu(p,\xi) =0$.
\item[(3)] For all $\delta\in(0,\delta_0)$ we can find $q\in G^c\cdot
  p$ such that $\mu(q) + \delta\nu(q)=0$.
\end{itemize}
In particular whether or not we can find $q\in G^c\cdot p$ such that
$\mu(q) + \delta\nu(q) =0$ is independent of $\delta\in
(0,\delta_0)$. 
\end{prop}
\begin{proof}
Let us write $S_1\subset\mathfrak{g}$ for the unit sphere. 
To prove $(1)\Rightarrow (2)$, note that $(1)$ implies that
$W_{\mu+\delta\nu}(p,\xi) > 0$ for all $\xi\not=0$. 
Suppose first that $W_\mu(p,\xi) < 0$
for some $\xi\in S_1$. This would
imply that for all sufficiently small $\delta > 0$ we have $W_{\mu +
  \delta\nu}(p,\xi) < 0$, which is a contradiction if $\delta_0$ is
chosen to be sufficiently small. So $(1)$ implies that $W_\mu(p,
\xi)\geqslant 0$ for all $\xi$. In addition if $\xi\not=0$, but
$W_\mu(p,\xi)=0$, then clearly we must have $W_\nu(p,\xi) > 0$. 

To prove $(2)\Rightarrow (3)$ note that the set of $\xi\in S_1$ 
for which $W_\mu(p,\xi)=0$ is a closed subset in $S_1$
by lower semicontinuity of $W$. It follows
that there is a $c_0 > 0$ such that 
\[ W_\nu(p,\xi) > c_0,\text{ for all $\xi\in S_1$ such that $W_\mu(p,\xi)=0$}.\]
The set of $\xi\in S_1$ for which $W_\nu(p,\xi) > c_0$ is an open set
$U$, whose complement in $S_1$ is compact. Again by lower
semicontinuity, $W_\mu(p,\xi) > c_1$ for some $c_1 > 0$ for all
$\xi\in S_1\setminus U$. The boundedness of $W_\nu(p,\xi)$ for $\xi\in
S_1$ then easily implies that for sufficiently small $\delta$ we have
\[ W_{\mu + \delta\nu}(p,\xi) > 0,\text{ for all }\xi\not=0. \]

The implication $(3)\Rightarrow (1)$ is immediate. 
\end{proof}

\section{The gluing theorem}\label{sec:gluing}
The goal of this section is to state and prove the main gluing theorem
that we will use, namely Theorem~\ref{thm:gluing}. 

\subsection{Preliminary discussion}\label{sec:preliminary}
We will use a technique very similar to that employed
in our earlier work~\cite{GSz10}. The first step is to use cutoff
functions to glue the extremal metric on $M$ to a model metric (the
Burns-Simanca metric) on the blowup of $\mathbf{C}^m$ at the origin,
scaled down by a factor of $\epsilon^2$. The gluing is performed on a
small annulus around the point $p$. This results in a metric
$\omega_\epsilon$ on $\Bl_pM$, whose scalar curvature is controlled,
since the Burns-Simanca metric is scalar flat. The problem of perturbing
$\omega_\epsilon$ to an extremal metric can then be written as 
finding a zero of a map
\begin{equation}\label{eq:F}
	F : C^\infty(\Bl_pM) \times \mathfrak{h} \to C^\infty(\Bl_pM),
\end{equation}
where $\mathfrak{h}$ is the space of Hamiltonian holomorphic vector fields on
$\Bl_pM$. For the exact form of $F$, see Section~\ref{sec:extequation} and
note that in practice we must work with various weighted H\"older
spaces instead of $C^\infty(\Bl_pM)$. 
The main technical difficulty in
constructing an extremal metric on the blowup $\Bl_pM$ is that in general
the space $\mathfrak{h}$ has lower dimension than the space
$\mathfrak{g}$ of Hamiltonian holomorphic vector fields on $M$.
The way this manifests itself in the analysis is that it is more
difficult to find a well-controlled right-inverse for the linearization
of $F$ as $\epsilon\to 0$.
In~\cite{GSz10} we overcome this problem by introducing a more general
operator of the form
\begin{equation}\label{eq:tildeF}
 \widetilde{F} : C^\infty(\Bl_pM) \times \mathfrak{g} \to
C^\infty(\Bl_pM),
\end{equation}
such that if $f\in\mathfrak{h}$, then $\widetilde{F}(\phi,f) =
F(\phi,f)$. So if we find a zero $\widetilde{F}
(\phi, f)=0$ with $f\in\mathfrak{h}$, then
we have an extremal metric. A right-inverse is not hard to construct for
the linearization of
$\widetilde{F}$, so we find a solution $(\phi_p, f_p)$. If we blow up
at a different point $q$, then we obtain a different pair $(\phi_q,
f_q)$. The crucial point is to compute the leading terms in $f_p$ as
 $\epsilon\to 0$. This can be done by finding better approximate
 solutions than $\omega_\epsilon$. In \cite{GSz10} we found that the
first non-trivial term is $\mu(p)$, using a technique similar to
\cite{APS06} to improve the approximate solution. 
A finite dimensional perturbation argument then shows
that if the vector field $\mu(p)$ vanishes at $p$, then for small
$\epsilon$ there is a point
$q\in G^c\cdot p$ such that $f_q$ vanishes at $q$. This gives us an
extremal metric on $\Bl_qM$, but $\Bl_qM$ is biholomorphic to $\Bl_pM$.  

To prove Theorem~\ref{thm:main} we need to find more terms in the
expansion of $f_p$ as $\epsilon\to 0$, so we need to 
find better approximate solutions. This involves some extra terms which we
have not considered in~\cite{GSz10}, see Section~\ref{sec:approx}.

\subsection{Possible generalizations}\label{sec:general}
We will now briefly discuss the
new difficulties that arise in trying to generalize our results. 
\subsubsection{The case $m=2$}
The main issue with the case when $m=2$ is that we need to compute
more terms in the expansion of $f$ in Theorem~\ref{thm:gluing}. This
can most easily be seen in the formula for the Futaki invariant on a
blowup in Corollary~\ref{cor:FutBl}, since the term involving
$\Delta h_v$ vanishes for $m=2$. The algebro-geometric formula has
been computed to more terms in \cite{GSz10}, and the problem is to
construct sufficiently good approximate solutions to obtain a
corresponding expansion of $f$. We believe that this should be
possible, but it needs a deeper analysis of the linearized problem
than what we have performed in the case $m > 2$. 

\subsubsection{More blowup points}
A more significant issue arises when we try to blow up more than one
point, in contrast to previous works~\cite{AP06,AP09,APS06,GSz10},
where the number of points made little difference. The new
complication in Theorem~\ref{thm:main} is that we need to perform a
gluing construction at the points $q_\delta$ in the $G^c$-orbit of
$p$, for which $\mu(q_\delta) +
\delta\Delta\mu(q_\delta)$ vanishes at $q_\delta$, for arbitrarily small
$\delta$. In the borderline case when $p$ is not relatively stable
with respect to $\mu$, the corresponding points $q_\delta$ will
approach the boundary of the $G^c$-orbit of $p$ as $\delta\to 0$. When
there is only one blowup point, then this is not a problem, since the
geometry of the manifolds $\Bl_{q_\delta}M$ is controlled as
$\delta\to 0$. When we blow up an $n$-tuple $(p_1,\ldots,p_n)$,
however, then as we approach the boundary of the orbit, some of the
points may approach each other. In this case the geometry of the
blowup is only controlled as long as the $n$-tuple stays away from the
``large diagonal'' in the $n$-fold product $M\times
M\times\ldots\times M$, where at least two points coincide. This means
that in order to use the same strategy as what we used in the proof of
Theorem~\ref{thm:main}, we would need to obtain results analogous to
Propositions~\ref{prop:inverse1} and \ref{prop:inverse2}, where the
norm of the inverse will now depend on the distance of our $n$-tuple
from the large diagonal in a suitable sense. Alternatively one may try
to make contact with the results on constructing cscK or extremal
metrics on iterated blowups, developed for K\"ahler surfaces in
LeBrun-Singer~\cite{LS93}, Rollin-Singer~\cite{RS07} and
Tipler~\cite{Tip11} for instance. 

\subsubsection{Extremal metrics} When $(M,\omega)$ is cscK, we were
able to obtain a sharp existence result for cscK metrics on the
blowups $\Bl_pM$. Many of the arguments can be adapted with little
difficulty to extremal metrics, however we were not able to show that
when the hypothesis of Theorem~\ref{thm:main} fails, then the blowups
$(\Bl_pM, [\omega_\epsilon])$ are relatively K-unstable for
sufficiently small $\epsilon$. The basic reason is that the inner
product of vector fields lifted to a blowup is not the same as their
inner product on $M$. The inner product enters in the definition of
relative stability, and in order to obtain a sharp existence result,
we would need a version of Theorem~\ref{thm:main} where the inner
product on $\mathfrak{g}$ is also perturbed in order to match with the
inner product of lifted vector fields. In practice this means that we
would need to compute more terms in the expansion of $f$, similarly to
the $m=2$ case. An alternative  approach would be to show
directly that the map $p\mapsto f_p$ above can itself be thought of as
a moment map for a suitable perturbed K\"ahler form on $M$ together
with a perturbed inner product on $\mathfrak{g}$, without necessarily
knowing the expansion explicitly. 

\subsection{Burns-Simanca metric}\label{sec:BurnsSimanca}
To obtain the first approximate solution $\omega_\epsilon$, we want
to glue the extremal
metric $\omega$ on $M$ to a rescaling of a 
suitable model metric on $\Bl_0\mathbf{C}^m$,
ie. on the blowup of $\mathbf{C}^m$ at the origin. This model metric is
a scalar flat metric found by Burns (see LeBrun~\cite{LeBrun88})
for $m=2$ and by
Simanca~\cite{Sim91}
for
$m\geqslant 3$. Away from the exceptional divisor it can be written in
the form
\[ \eta = \ddbar\left( \frac{1}{2}|w|^2 +\psi(w) \right),\]
where $w=(w_1,\ldots,w_m)$ are standard coordinates on $\mathbf{C}^m$.
The function $\psi$ can be found by solving an ODE. We will need the
following result from Gauduchon~\cite{Gau12}
 about the asymptotics of $\psi$.  Note that we use $\ddbar$ as
 opposed to $dd^c$ which introduces a factor of 2 in our formula, and
 also our normalization of the volume of the exceptional divisor is
 different. 
\begin{lem}\label{lem:BSmetric}
	If $m\geqslant 3$ then the K\"ahler potential for a suitable
	scaling of the
	Burns-Simanca metric 
	\[\eta=\ddbar(|w|^2/2 + \psi(w))\]
	satisfies
	\begin{equation}\label{eq:metric}
		\psi(w) = -\frac{1}{2\pi^{m-1}(m-2)}|w|^{4-2m} + 
                d_1|w|^{2-2m} + d_2|w|^{6-4m}
                + O(|w|^{4-4m}),
	\end{equation}
	where $d_1 > 0$. The scaling of $\eta$ is such that the
        exceptional divisor has volume $\frac{1}{(m-1)!}$.  
\end{lem}
The important aspect of this result for us is the formula for the
coefficient of the first term, and the sign of the second. 
Note
that the scaling is chosen in such a way that if we construct
$\omega_\epsilon$ as in Section~\ref{sec:omegaepsilon}, then we
end up with a metric in the class $\pi^*[\omega] - \epsilon^2[E]$. 

\subsection{The metric $\omega_\epsilon$ on $\Bl_pM$}
\label{sec:omegaepsilon}
Suppose as before that $\omega$ is an extremal K\"ahler metric on $M$.
Let $X_\mathbf{s}$ be the Hamiltonian vector field corresponding to the
scalar curvature $\mathbf{s}(\omega)$.  
Write $G$ for the Hamiltonian isometry group of
$(M,\omega)$, so the Lie algebra $\mathfrak{g}$ of $G$ consists
of holomorphic Killing fields with zeros. 

Choose a point $p\in M$ where the vector
field $X_\mathbf{s}$ vanishes, and let $T\subset G_p$ be a torus
fixing $p$ whose Lie algebra contains $X_{\mathbf{s}}$.
Let $H\subset G$ consist of the elements
commuting with $T$ and let us write $\overline{\mathfrak{h}}
\subset C^{\infty}(M)$
for the space of Hamiltonian functions of vector fields in the Lie
algebra of $H$. Note that $\overline{\mathfrak{h}}$ contains
the constants as well. 
Let us also write $\overline{\mathfrak{t}}\subset\overline{\mathfrak{h}}$ 
for the Hamiltonian functions corresponding to the subgroup $T\subset
H$. 

Given a small parameter $\epsilon > 0$,
we will construct an approximate solution to our problem 
on $\Bl_pM$ in the K\"ahler class $\pi^*[\omega] - \epsilon^2[E]$.
For simplicity assume that the exponential map is defined on the unit
ball in the tangent space $T_pM$ (if not, we can scale up the metric
$\omega$). Choose local normal
coordinates $z$ near $p$ such that the group $T$
acts by unitary transformations on the unit ball $B_1$ around $p$ (this
is possible by linearizing the action, see Bochner-Martin~\cite{BM48}
Theorem 8). In
these coordinates we can write
\[ \omega = \ddbar\big(|z|^2/2 + \phi(z)\big),\]
where $\phi=O(|z|^4)$. At the same time 
the Burns-Simanca metric from Section~\ref{sec:BurnsSimanca} has the form
\[ \eta = \ddbar\big(|w|^2/2 + \psi(w)\big).\]
We glue $\epsilon^2\eta$ to $\omega$ using a cutoff function in the
annulus $B_{2r_\epsilon}\setminus B_{r_\epsilon}$ in $M$, where
$r_\epsilon=\epsilon^\alpha$ for some $\alpha < 1$ to be chosen
later. Under the gluing we will have $z = \epsilon w$. 
To do this, let $\gamma:\mathbf{R}\to[0,1]$ be smooth such
that $\gamma(x)=0$ for $x < 1$ and $\gamma(x)=1$ for $x > 2$. Define
\[ \gamma_1(r) = \gamma(r/r_\epsilon), \] and
write $\gamma_2=1-\gamma_1$. Then for small $\epsilon$
we can define a K\"ahler metric
$\omega_\epsilon$ on $\Bl_pM$ which on the annulus $B_1\setminus
B_\epsilon$ is given by 
\[
	\omega_\epsilon = \ddbar\left( \frac{|z|^2}{2} 
	+ \gamma_1(|z|)\phi(z) 
+ \gamma_2(|z|) \epsilon^2\psi(\epsilon^{-1}
z)\right).\]
Moreover outside $B_{2r_\epsilon}$ the metric $\omega_\epsilon=\omega$ while
inside the ball $B_{r_\epsilon}$ we have
$\omega_\epsilon=\epsilon^2\eta$. 
Note that the action of $T$ lifts to $\Bl_pM$ 
giving biholomorphisms,
and that $\omega_\epsilon$ is $T$-invariant.

In order to define the operator $\widetilde{F}$ from
\eqref{eq:tildeF}, we need to lift elements in
$\overline{\mathfrak{h}}$ to $\Bl_pM$. 

\begin{defn}\label{def:lifting} We define a linear map 
\[ \l : \overline{\mathfrak{h}} \to C^\infty(\Bl_pM) \]
as follows. Decompose $\overline{\mathfrak{h}}$ into a direct sum
$\overline{\mathfrak{h}} =
\overline{\mathfrak{t}}\oplus\mathfrak{h}'$ in such a way that each
function in $\mathfrak{h}'$ vanishes at $p$. Each
$f\in\overline{\mathfrak{t}}$ corresponds to a holomorphic Hamiltonian
vector field $X_f$ on $M$ vanishing at $p$. We then define $\l(f)$ to be
the Hamiltonian function with respect to $\omega_\epsilon$ 
of the holomorphic lift of $X_f$ to $\Bl_pM$, normalized so that $f =
\l(f)$ outside $B_1$. For $f\in\mathfrak{h}'$ we define $\l(f) =
\gamma_1 f$ near $p$ using the cutoff function $\gamma_1$ from
above. We can then think of this $\l(f)$ as a function on
$\Bl_pM$. Finally we can extend $\l$ to all of
$\overline{\mathfrak{h}}$ by linearity. 
\end{defn}
Note that in contrast to \cite{GSz10} we are not assuming that
$T\subset G_p$ is a maximal torus. This is necessary for technical
reasons, namely we will want to be able to work with all $T$-invariant
points at the same time. On the other hand
it implies that even if $f\in\overline{\mathfrak{h}}$
corresponds to a holomorphic vector field vanishing at $p$, its lift
$\l(f)$ will not give rise to a holomorphic vector field, unless
$f\in\overline{\mathfrak{t}}$.

We will also need lifts corresponding
to metrics other than $\omega_\epsilon$. If $\Omega = \omega_\epsilon
+ \ddbar\Phi$ and $\Phi$ is $T$-invariant, then we define
\[ \l_\Omega(f) = \l(f) + \frac{1}{2} \nabla\Phi\cdot \nabla f.\]
If $f\in\overline{\mathfrak{t}}$, then $\l_\Omega(f)$ is a Hamiltonian
function for the vector field $X_f$, with respect to $\Omega$. In
particular this has the following consequence.
\begin{lem}\label{lem:intl}
If $f\in\overline{\mathfrak{t}}$, then 
\[ \int_{\Bl_pM} \l_\Omega(f)\,\Omega^m = \int_{\Bl_pM} \l(f)\,\omega_\epsilon^m.\]
\end{lem}
\begin{proof}
This can be checked by using $\Omega_t = \omega_\epsilon +
t\ddbar\Phi$, and differentiating with respect to $t$. Alternatively,
in the algebraic case,
note that the integral of the Hamiltonian function of a holomorphic
vector field can be computed algebro-geometrically, so in particular
it is independent of the metric. 
\end{proof}

\subsection{The extremal metric equation}\label{sec:extequation}
We will now write down what the operators $F$ and $\widetilde{F}$ in
\eqref{eq:F} and \eqref{eq:tildeF} look
like. We have a torus $T\subset G$ fixing $p$, and we seek a
$T$-invariant function $\phi$ on $\Bl_pM$ such that $\omega_\epsilon +
\ddbar\phi$ is an extremal metric. 

We need the following which can also be found in
\cite{APS06}, \cite{GSz10}.
\begin{lem}\label{lem:basiceqn}
	Suppose that $\phi\in C^\infty(\Bl_pM)^T$ and 
	$f\in\overline{\mathfrak{t}}$ such that
	\begin{equation}
		\label{eq:basiceqn}
		\mathbf{s}
		(\omega_\epsilon+\ddbar\phi) - \frac{1}{2}\nabla
		\l(f)\cdot\nabla\phi = \l(f),
	\end{equation}
	where the gradient and inner product are computed with respect
	to the
	metric $\omega_\epsilon$.
	Then $\omega_\epsilon + \ddbar\phi$ is an extremal metric. 
\end{lem}

In order to solve Equation~(\ref{eq:basiceqn}) as a perturbation problem,
we will write it in the form
\begin{equation}\label{eq:2}
	\mathbf{s}(\omega_\epsilon+\ddbar\phi) -
	\frac{1}{2}\nabla\l(\mathbf{s}+f)\cdot\nabla\phi =
	\l(\mathbf{s}+f),
\end{equation}
where $\mathbf{s}\in\mathfrak{t}$ is the scalar curvature of the
extremal metric $\omega$. 
The advantage of this is that we now
seek $\phi$ and $f$ which are small, or in other words, setting
$\phi=0$ and $f=0$ we get an approximate solution to the
equation. 

For any K\"ahler metric $\Omega$ let us define the operators
$L_{\Omega}$ and $Q_{\Omega}$ by
\begin{equation}\label{eq:scal}
	\mathbf{s}(\Omega + \ddbar\phi) = \mathbf{s}
	(\Omega) +
	L_{\Omega}(\phi) + Q_{\Omega}(\phi),
\end{equation}
where $L$ is the linearized operator. A simple computation shows that
\[ L_{\Omega}(\phi) = -\Delta^2_{\Omega}\phi -
\mathrm{Ric}(\Omega)^{i\bar j}\phi_{i\bar j},\]
and analysing this operator will be crucial later on. Note that we are using the
complex Laplacian here which is half of the usual Riemannian one.  
The
linear operator appearing in the linearization of
Equation~(\ref{eq:2}) is then 
\begin{equation}
	\label{eq:linop}
	(\phi,f) \mapsto  L_{\Omega_\epsilon}(\phi) -
	\frac{1}{2}\nabla\l(\mathbf{s}) \cdot
	\nabla \phi - \l(f),
\end{equation}
which is closely related to the Lichnerowicz operator that we will
discuss in Section~\ref{sec:Lichnerowicz}.

We can now state the main gluing result that we will prove, which
corresponds to finding a zero of the operator $\widetilde{F}$ in
\eqref{eq:tildeF}. The proof of this theorem will appear at the end of
Section~\ref{sec:nonlinear}. 
\begin{thm}\label{thm:gluing}
Fix a torus $T\subset G$ such that $\mathbf{s}\in\overline{\mathfrak{t}}$. 
There are constants $\epsilon_0 > 0$ and $\kappa > 2m$
with the following property. Suppose that $p\in M$ is a fixed point of
$T$. For every $\epsilon\in(0,\epsilon_0)$ we can find $u\in
C^{4,\alpha}(\Bl_pM)^T$ and $f\in\overline{\mathfrak{h}}$ satisfying the
equation
\begin{equation}\label{eq:maineq}
 \mathbf{s}(\omega_\epsilon + \ddbar u) = \l_{\omega_\epsilon + \ddbar
   u}(f) = \l(f) + \frac{1}{2}\nabla\l(f)\cdot \nabla u. 
\end{equation}
In addition the element $f\in\overline{\mathfrak{h}}$ has an expansion
\[ f = \mathbf{s} + C - \epsilon^{2m-2}\left(c_1 -
  \frac{\epsilon^2}{m!}\mathbf{s}(p) \right)\mu(p) - \epsilon^{2m}c_2\Delta\mu(p) +
O(\epsilon^\kappa), \]
for some $\kappa > 2m$, where $C$ is a constant depending on
$\epsilon$, and $c_1,c_2 > 0$. 
\end{thm}

Assuming this result, we can prove Theorem~\ref{thm:main}. 
\begin{proof}[Proof of Theorem~\ref{thm:main}] 
By replacing $p$ with a different point in its $G^c$-orbit, we can
assume that the stabilizer $G_p$ is a maximal compact subgroup of the
complex stabilizer $G^c_p$. Since the scalar curvature
$\mathbf{s}$ is $G$-invariant, the new point will also be a
critical point of $\nabla\mathbf{s}$. 

Let $T\subset G_p$ be a maximal torus. Then $T^c\subset G^c_p$ is also
a maximal torus, and writing $\mathfrak{t}$ for the Lie algebra, we
have $\mathbf{s}\in\mathfrak{t}$. Let $H\subset G$ consist of the
elements of $G$ commuting with $T$. As above, we write $\overline{\mathfrak{h}}$
for functions on $M$ whose Hamiltonian vector fields are in
$\mathfrak{h}$. We apply
Theorem~\ref{thm:gluing} to the set $M^T$ of $T$-invariant points in
$M$. For every $q\in M^T$ and $\epsilon\in(0,\epsilon_0)$ we obtain a
$T$-invariant function $u_{q,\epsilon}\in C^{4,\alpha}(\Bl_qM)$ and
$f_{q,\epsilon}\in\overline{\mathfrak{h}}$ such that if
$f_{q,\epsilon}\in\overline{\mathfrak{t}}$, then we have an extremal
metric on $\Bl_qM$. From Theorem~\ref{thm:gluing} we know that
\[ f_{q,\epsilon} = \mathbf{s} + C - \epsilon^{2m-2}\left(c_1 -
  \frac{\epsilon^2}{m!} \mathbf{s}(q)\right)\mu(q) - \epsilon^{2m}c_2
\Delta\mu(q) + O(\epsilon^\kappa), \]
where $\kappa > 2m$, $C$ is a constant depending on
$\epsilon$, and $c_1,c_2 > 0$. Let us write
\[ \begin{aligned}
\mu_\epsilon(q) &= -\epsilon^{2-2m}\left(c_1 -
  \frac{\epsilon^2}{m!}\mathbf{s}(q)\right)^{-1} (f_{q,\epsilon} -
\mathbf{s} - C) \\ &= \mu(q) + \epsilon^2\frac{c_2}{c_1}\Delta\mu(q) +
O(\epsilon^{\kappa'}),
\end{aligned} \]
where $\kappa' > 2$. Then $f_{q,\epsilon}\in\overline{\mathfrak{t}}$
if and only if $\mu_{\epsilon}(q)\in\overline{\mathfrak{t}}$. 

We will now apply Theorem~\ref{thm:perturb} to the action of $G^c_{T^\perp}$ on
$M$, where $G_{T^\perp}\subset H$ is the group introduced in
Section~\ref{sec:finitedim}. The corresponding 
moment maps $\mu_{T^\perp}$ and $\Delta\mu_{T^\perp}$ 
are $\mu$, $\Delta\mu$ projected to
$\mathfrak{g}_{T^\perp}$. From the assumption of Theorem~\ref{thm:main},
together with Proposition~\ref{prop:GT}, we know that for some
sufficiently small $\delta_0 > 0$ there is a point $q$ in the
$G^c_{T^\perp}$-orbit of $p$ such that 
\[ \mu(q) + \delta_0\Delta\mu(q) \in \mathfrak{g}_q. \]
Moreover, since $T^c$ is a maximal torus in $G_p^c$, and $G_{T^\perp}$
commutes with $T$, we have $\mathfrak{g}_q = \mathfrak{t}$.  It
follows that
\[ \mu_{T^\perp}(q) + \delta_0\Delta\mu_{T^\perp}(q) = 0. \]
Proposition~\ref{prop:alldelta} and 
Theorem~\ref{thm:perturb} now imply that for all sufficiently small
$\epsilon > 0$ we can find $q\in G^c_{T^\perp}\cdot p$ such that 
\begin{equation}\label{eq:mueps}
 \mathrm{pr}_{\mathfrak{g}_{T^\perp}}\mu_\epsilon(q) = 0. 
\end{equation}
By construction, $\mu_\epsilon(q) \in \overline{\mathfrak{h}}$,
i.e. $\mu_\epsilon(q)$ commutes with $T$, so \eqref{eq:mueps} implies
that $\mu_\epsilon(q) \in \overline{\mathfrak{t}}$. This implies that
$f_{q,\epsilon}\in\overline{\mathfrak{t}}$, and so we have obtained an
extremal metric on $\Bl_qM$, in the class $\pi^*[\omega] -
\epsilon^2[E]$.  Since $q\in G^c\cdot p$, the manifold
$\Bl_qM$ is biholomorphic to $\Bl_pM$.
\end{proof}

\subsection{The Lichnerowicz operator}\label{sec:Lichnerowicz}
For any K\"ahler metric
$\Omega$ on a manifold $X$ we have the operator
\[ \mathcal{D}_{\Omega}:C^\infty(X) \to \Omega^{0,1}(T^{1,0}X),\]
given by $\mathcal{D}(\phi) = \dbar\nabla^{1,0}\phi$ where $\dbar$ is
the natural $\dbar$-operator on the holomorphic tangent bundle. 
The Lichnerowicz operator is then the fourth order operator
\[ \mathcal{D}^*_{\Omega}\mathcal{D}_{\Omega} 
: C^\infty(X)\to C^\infty(X),\]
whose significance is that the kernel consists of precisely those 
functions whose gradients are holomorphic vector fields.
The relation to the operator in Equation~(\ref{eq:linop}) is that a
computation (see eg. LeBrun-Simanca~\cite{LS94}) shows that
\begin{equation}\label{eq:Lichn}
	\mathcal{D}^*_{\Omega}\mathcal{D}_{\Omega}(\phi)
	= -L_{\Omega}(\phi) + \frac{1}{2}\nabla
	\mathbf{s}(\Omega)\cdot\nabla\phi.
\end{equation}
When comparing this to Equation~\eqref{eq:linop}, note that in
general $\mathbf{s}(\Omega_\epsilon)$ is not equal to
$\l(\mathbf{s})$. The difference will be sufficiently small though.

\subsection{The Lichnerowicz operator on weighted spaces}
As in Arezzo-Pacard~\cite{AP06,AP09}, Arezzo-Pacard-Singer~\cite{APS06}
and also~\cite{GSz10}, we need to study the invertibility of the
linearized operator between suitable weighted H\"older spaces on the
blowup $\Bl_pM$. First we
need to understand the behaviour the Lichnerowicz operator on weighted
spaces on the manifolds $M\setminus\{p\}$ and $\Bl_p\mathbf{C}^m$, and
then obtain results about the blowup by ``gluing'' these spaces. This
section is parallel to Section 5.1 in~\cite{GSz10}, but we need slightly
different results. 

Let us first consider $M_p=M\setminus\{p\}$ with the metric $\omega$. For
functions $f:M_p \to\mathbf{R}$ we define the weighted norm
\[ \Vert f\Vert_{C^{k,\alpha}_\delta(M_p)} = \Vert
f\Vert_{C^{k,\alpha}_\omega(M\setminus B_{1/2})}
+ \sup_{r < 1/2} r^{-\delta}
\Vert f\Vert_{C^{k,\alpha}_{r^{-2}\omega}(B_{2r}
\setminus B_r)}.\]
Here the subscripts $\omega$ and $r^{-2}\omega$ indicate the metrics
used for computing the corresponding norm. The weighted space
$C^{k,\alpha}_\delta(M_p)$ consists of functions on
$M\setminus\{p\}$ which are locally in $C^{k,\alpha}$ and whose
$\Vert\cdot\Vert_{C^{k,\alpha}_\delta}$ norm is finite.

We need the following result, which is Proposition 17 in~\cite{GSz10}.
As before, we have a torus $T\subset G$
fixing the point $p$, and $H\subset G$ is the centralizer of $T$. 
\begin{prop}\label{prop:Mp}
	If $\delta < 0$, $\delta$ is not an integer,   
	and $\alpha\in(0,1)$, then the operator 
	\[ \begin{aligned} 
		C^{k,\alpha}_{\delta}(M_p)^T \times
		\overline{\mathfrak{t}}
		&\to C^{k-4,\alpha}_{\delta-4}(M_p)^T \\
		(\phi, f) \mapsto \mathcal{D}^*_\omega
		\mathcal{D}_\omega \phi - f 
	\end{aligned}\]
	has a bounded right-inverse. 
\end{prop}

Let us turn now to the manifold $\Bl_0\mathbf{C}^m$ with the
Burns-Simanca metric $\eta$. The relevant weighted H\"older norm is now
given by
\[ \Vert f\Vert_{C^{k,\alpha}_\delta(\Bl_0\mathbf{C}^m)} = \Vert
f\Vert_{C^{k,\alpha}_\eta(B_2)} + \sup_{r > 1} r^{-\delta} \Vert
f\Vert_{C^{k,\alpha}_{r^{-2}\eta}(B_{2r}\setminus B_r)}.\]
Here we abused notation slightly by writing $B_r\subset Bl_0
\mathbf{C}^m$ for the set where $|z|<r$ (ie. the pullback of the
$r$-ball in $\mathbf{C}^m$ under the blowdown map).

The following is Proposition 18 from~\cite{GSz10}. 
\begin{prop}\label{prop:Bl0}
If $\delta > 4-2m$ the operator
\[ \begin{aligned}
  C^{k,\alpha}_\delta(\Bl_0\mathbf{C}^m) &\to
  C^{k-4,\alpha}_{\delta-4}(\Bl_0\mathbf{C}^m) \\
\phi &\mapsto \mathcal{D}^*_\eta\mathcal{D}_\eta\phi 
\end{aligned}\]
has a bounded right inverse.

If $\delta\in(3-2m,4-2m)$, let $\chi$ be a compactly supported
function on $\Bl_0\mathbf{C}^m$ with non-zero integral. The operator
\[ \begin{aligned}
    C^{k,\alpha}_\delta(\Bl_0\mathbf{C}^m)\times\mathbf{R} &\to
    C^{k-4,\alpha}_{\delta-4}(\Bl_0\mathbf{C}^m), \\
    (\phi, t) &\mapsto \mathcal{D}^*_\eta\mathcal{D}_\eta(\phi) +
    t\chi
  \end{aligned}\]
has a bounded right inverse. 
\end{prop}

\subsection{Weighted spaces on $\Bl_pM$}
We will need to do analysis on the blown-up manifold $\Bl_pM$ endowed
with the approximately extremal metric $\omega_\epsilon$. For this
we define the following weighted spaces, which are simply glued versions
of the above weighted spaces on $M\setminus\{p\}$ and
$\Bl_p\mathbf{C}^m$.

We define the weighted H\"older norms $C^{k,\alpha}_\delta$ by
\[ \Vert f\Vert_{C^{k,\alpha}_\delta} = \Vert
f\Vert_{C^{k,\alpha}_{\omega}(M\setminus B_1)} + \sup_{\epsilon\leqslant
r\leqslant 1/2} r^{-\delta}\Vert
f\Vert_{C^{k,\alpha}_{r^{-2}\omega_\epsilon}(B_{2r}\setminus B_r)} +
\epsilon^{-\delta}\Vert f\Vert_{C^{k,\alpha}_{\eta}(B_\epsilon)}. \]

The subscripts indicate the metrics used to compute the relevant norm.
This is a glued version of the two spaces defined in the previous
section
in the following sense. If $f\in C^{k,\alpha}(\Bl_pM)$ and we think of $\Bl_pM$
as a gluing of $M\setminus\{p\}$ and $\Bl_0\mathbf{C}^m$ then $\gamma_1f$
and $\gamma_2f$ can naturally be thought of as functions on
$M\setminus\{p\}$ and $\Bl_0\mathbf{C}^m$ respectively. Then the norm
$\Vert f\Vert_{C^{k,\alpha}_\delta(\Bl_pM)}$ is comparable to
\[ \Vert \gamma_1f\Vert_{C^{k,\alpha}_\delta(M_p)} + 
\epsilon^{-\delta}\Vert\gamma_2f\Vert_{C^{k,\alpha}_\delta
(\Bl_0\mathbf{C}^m,\eta)}.\]

Another way to think about the norm is that if $\Vert
f\Vert_{C^{k,\alpha}_\delta}\leqslant c$ then $f$ is in
$C^{k,\alpha}(\Bl_pM)$ and also for $i\leqslant k$ we have
\[ \begin{gathered}
	|\nabla^i f| \leqslant c\,\,\text{ for } r\geqslant 1\\
	|\nabla^i f| \leqslant cr^{\delta-i}\,\,\text{ for } \epsilon
	\leqslant r \leqslant 1\\
	|\nabla^i f| \leqslant c\epsilon^{\delta-i}
	\,\,\text{ for } r\leqslant\epsilon.
\end{gathered}\]
The norms here are computed with respect to the metric
$\omega_\epsilon$, and note that on $B_\epsilon$ we have
$\omega_\epsilon=\epsilon^2\eta$. We will often use the following to
compare the different weighted norms:
\[ \Vert f\Vert_{C^{k,\alpha}_\delta} \leqslant \begin{cases} \Vert
  f\Vert_{C^{k,\alpha}_{\delta'}},\text{ if }\delta' > \delta,\\
  \epsilon^{\delta' - \delta}\Vert
  f\Vert_{C^{k,\alpha}_{\delta'}},\text{ if }\delta' < \delta.
\end{cases} \]

Sometimes we will restrict the norm to subsets such as
$C^{k,\alpha}_\delta(M\setminus B_{r_\epsilon})$ and
$C^{k,\alpha}_\delta(B_{2r_\epsilon})$. 
A crucial property of these weighted norms is that
\begin{equation}\label{eq:cutoff}
	\Vert \gamma_i\Vert_{C^{4,\alpha}_0} \leqslant c 
\end{equation}
for some constant $c$ independent of $\epsilon$, where $\gamma_i$ are
the cutoff functions from Section~\ref{sec:omegaepsilon}.

In addition we need
the following lemma about lifting elements of
$\overline{\mathfrak{h}}\subset C^\infty(M)$ to 
$C^\infty(\Bl_pM)$ according to Definition~\ref{def:lifting}. 

\begin{lem}\label{lem:lb}
	For any $f\in\overline{\mathfrak{h}}$ its lifting satisfies
\[
\begin{aligned}
   \Vert \l(f) \Vert_{C^{1,\alpha}_0} &\leqslant c|f|, \\
   \Vert \nabla\l(f)\Vert_{C^{1,\alpha}_0} &\leqslant c|f|,
\end{aligned}\]
for some constant $c$ independent of $\epsilon$.
Here $|\cdot|$ is any fixed norm on $\overline{\mathfrak{h}}$.
\end{lem}
\begin{proof}
	Recall that we defined the lifting using a decomposition
	$\overline{\mathfrak{h}}=\overline{\mathfrak{t}}
	\oplus\mathfrak{h}'$, where the
	functions in $\mathfrak{h}'$ vanish at $p$. Suppose
	first that $f\in\mathfrak{h}'$. Since $f$ vanishes at $p$, we
	have
	\[ \Vert f\Vert_{C^{1,\alpha}_1(M_p)}\leqslant c|f|,\]
	where $c$ is independent of $f$. It follows from the
	multiplication properties of weighted spaces and
	(\ref{eq:cutoff}) that
	\[ \Vert \l(f)\Vert_{C^{1,\alpha}_1} \leqslant c|f|,\]
	from which the required inequalities follow.

	Now suppose that $f\in\overline{\mathfrak{t}}$, 
	and write $X_f$ for the
	holomorphic vector field on $M$ corresponding to $f$. The
        result is clearly true for constants, so we can assume that
        $f$ vanishes at $p$. On the
	ball $B_{r_\epsilon}\subset M$,
	the action of $X_f$ is given by unitary transformations,
	generated by a matrix $A$, say.
	Outside $B_{r_\epsilon}$ the vector field is unchanged and the 
	metrics $\omega$ and $\omega_\epsilon$ are uniformly equivalent.
	Inside $B_{r_\epsilon}$ the metric $\omega_\epsilon$ is
	uniformly equivalent to $\epsilon^2\eta$. It is more convenient
	to work with $\eta$, since that is a fixed metric, and we can
	then scale back depending on $\epsilon$. Let $f_\eta$ be the
	Hamiltonian function of $X_f$ with respect to $\eta$, so
	$f= \epsilon^2 f_\eta$. In terms of $\eta$ we are working on 
	the ball $B_{R_\epsilon}$, and outside $B_1$ the Hamiltonian
	$f_\eta$ is given by a quadratic function depending on $A$. It
	follows that we have pointwise bounds
	\[|\nabla^i f_\eta(x)|_\eta\leqslant C_ir(x)^{2-i}|A|,\]
	where $r(x)=1$ inside $B_1$, and $r(x)$ is the distance from the
	exceptional divisor outside $B_1$. We can choose the norm $|A|$
	to coincide with the norm $|f|$ chosen on the finite dimensional
	vector space $\overline{\mathfrak{h}}$. Rescaling this
	inequality, together with what we already know outside
	$B_{r_\epsilon}$, we
	get $\Vert \l(f)\Vert_{C^{k,\alpha}_2}\leqslant C|f|$, which implies
	the results that we want. 
\end{proof}

\subsection{The linearized operator on $\Bl_pM$}
We now begin studying the linearized operator on $\Bl_pM$, in
terms of the weighted spaces introduced in the previous section. 
The constants that appear below will be independent of
$\epsilon$ unless the dependence is made explicit. 

Recall that for any metric $\Omega$ we write
\[ L_{\Omega}(\phi) = -\Delta_{\Omega}^2 \phi -
\mathrm{Ric}(\Omega)^{i\bar j} 
\phi_{i\bar j}.\]
We want to first study how this varies as we change the
metric.
For this we have the following, which is
 Proposition 20 from~\cite{GSz10}.
\begin{prop}\label{prop:linear}
	Suppose that $\delta < 0$. 
	There exist constants $c_0, C>0$ such that if 
	$\Vert\phi\Vert_{C^{4,\alpha}_2} < c_0$
	then  	
	\[ \Vert L_{\omega_\phi}(f) -
	L_{\omega_\epsilon}(f)\Vert_{C^{0,\alpha}_{\delta-4}} \leqslant
	C\Vert\phi\Vert_{C^{4,\alpha}_2}\Vert
	f\Vert_{C^{4,\alpha}_\delta},\]
	where $\omega_\phi=\omega_\epsilon + \ddbar\phi$.
\end{prop}

One consequence is an estimate for the nonlinear operator
$Q_{\omega_\epsilon}$ in the formula
\begin{equation}\label{eq:Q}
	\mathbf{s}(\omega_\epsilon + \ddbar\phi) =
\mathbf{s}(\omega_\epsilon) + L_{\omega_\epsilon}(\phi) +
Q_{\omega_\epsilon}(\phi).
\end{equation}
The following is Lemma 21 in~\cite{GSz10}. 

\begin{lem}\label{lem:Q}
Suppose that $\delta < 0$. There exists a $c_0 >
0$ such that if 
\[ \Vert \phi\Vert_{C^{4,\alpha}_2}, \Vert \psi
\Vert_{C^{4,\alpha}_2} \leqslant c_0, \]
then 
\[ \Vert Q_{\omega_\epsilon}(\phi) -
Q_{\omega_\epsilon}(\psi)\Vert_{C^{0,\alpha}_{\delta-4}} \leqslant
C(\Vert\phi\Vert_{C^{4,\alpha}_2} + \Vert
\psi\Vert_{C^{4,\alpha}_2})
\Vert \phi-\psi\Vert_{C^{4,\alpha}_\delta}. \]
\end{lem}

We will need one further result, which was not used in~\cite{GSz10}.
\begin{lem}\label{lem:Q2} Suppose that $\omega = \omega_\epsilon +
  \ddbar\phi$, and
  \[ \Vert\phi\Vert_{C^{4,\alpha}_2},\Vert\psi\Vert_{C^{4,\alpha}_2}
  \leqslant c_0, \]
  for some sufficiently small $c_0$. Then 
  \begin{equation}\label{eq:Qdiff}
    \Vert Q_{\omega}(\psi) -
Q_{\omega_\epsilon}(\psi)\Vert_{C^{0,\alpha}_{\delta-4}} 
\leqslant C\Vert
\phi\Vert_{C^{4,\alpha}_2} \Vert\psi\Vert_{C^{4,\alpha}_2}
\Vert\psi\Vert_{C^{4,\alpha}_{\delta}}.
\end{equation}
\end{lem}
\begin{proof}
Let us write $g$ for a metric, and $g+h$ for a small perturbation,
thought of as matrices in local coordinates. We
can write schematically
\[ \begin{aligned} \mathbf{s}(g + h) &=
  (g+h)^{-1}\partial^2\log\det(g+h) \\
  &= g^{-1}(I + g^{-1}h)^{-1}\partial^2(\log\det g + \log\det (I +
  g^{-1}h)),
\end{aligned} \]
where $I$ is the identity matrix. Expanding in power series, we find
that $Q$ is of the form
\[ Q_g(h) = \sum_{i=0}^2 g^{-1}\big[\partial^i(g^{-1}h)^2\big] F_i(g^{-1}h), \]
where the $F_i$ are power series. In order to estimate $Q_{g_1}(h) -
Q_{g_2}(h)$ it is enough to consider a typical term, for instance
\[ g_1^{-1}\big[ \partial^2(g_1^{-1}h)^2\big] (g_1^{-1}h)^l - 
g_2^{-1}\big[ \partial^2(g_2^{-1}h)^2\big] (g_2^{-1}h)^l, \]
for some $l\geqslant 0$. In our situation $h=\ddbar\psi$, and we have
\[ \begin{aligned}
  \Vert g_1^{-1} - g_2^{-1}\Vert_{C^{2,\alpha}_0} &\leqslant
  C\Vert\phi\Vert_{C^{4,\alpha}_2} \leqslant Cc_0, \\
  \Vert g_j^{-1}h\Vert_{C^{2,\alpha}_0} &\leqslant
  C\Vert\psi\Vert_{C^{4,\alpha}_2} \leqslant Cc_0, \\
  \Vert g_j^{-1}h\Vert_{C^{2,\alpha}_{\delta-2}} &\leqslant
  C\Vert\psi\Vert_{C^{4,\alpha}_\delta},
  \end{aligned}\]
for $j=1,2$. From this it is a straightforward calculation to check
the estimate \eqref{eq:Qdiff}. 
\end{proof}

The heart of the matter is to understand the invertibility
of the linearized operator of our problem on $\Bl_pM$. The following is
Proposition 22 from~\cite{GSz10}. 

\begin{prop}\label{prop:inverse1}
	For sufficiently small $\epsilon$ and $\delta\in(4-2m,0)$
	the operator 
	\[ \begin{gathered}
		G_1 : (C^{4,\alpha}_\delta)^T \times \overline{\mathfrak{h}} \to
		(C^{0,\alpha}_{\delta-4})^T \\ 
		(\phi, f) \mapsto L_{\omega_\epsilon}(\phi) - 
		\frac{1}{2}\nabla\mathbf{s}(\omega_\epsilon)\cdot
		\nabla\phi - \l(f)
	\end{gathered}\]
	has a right inverse $P_1$, with the operator norm $\Vert P_1\Vert <
	C$ for some constant $C$ independent of $\epsilon$.
\end{prop}

We will need a slight variation of this result as well, dealing with
weights in the range $(3-2m,4-2m)$. One can easily obtain a result for
$\delta\in (3-2m,4-2m)$ from the preivous proposition, but we will
only have a bound of the form $C\epsilon^{\delta-(4-2m)}$ for the
inverse. It turns out that if we restrict the range to functions with
zero mean, we can obtain an inverse with norm bounded independent of
$\epsilon$. 

\begin{prop}\label{prop:inverse2}
	Let us write $(C^{0,\alpha}_{\delta-4})^T_0$ for the elements in
	$(C^{0,\alpha}_{\delta-4})^T$ which have zero mean on $\Bl_pM$,
        and $\overline{\mathfrak{h}}_0$ for the elements
        $f\in\overline{\mathfrak{h}}$ such that $\l(f)$ has zero mean
        on $\Bl_pM$
	with respect to $\omega_\epsilon$. 
	For sufficiently small $\epsilon$, and $\delta\in(3-2m,4-2m)$,
	the operator 
	\[ \begin{gathered}
		G_2 : (C^{4,\alpha}_\delta)^T \times \overline{\mathfrak{h}}_0 \to
		(C^{0,\alpha}_{\delta-4})^T_0 \\ 
		(\phi, f) \mapsto L_{\omega_\epsilon}(\phi) -
		\frac{1}{2}\nabla\mathbf{s}(\omega_\epsilon)\cdot
		\nabla\phi
		- \l(f)
	\end{gathered}\]
	has a right inverse $P_2$, with the operator norm $\Vert P_2\Vert <
	C$ for some constant $C$ independent of $\epsilon$.
\end{prop}
\begin{proof}
  The proof is very similar to the proof of Proposition 22 in
  \cite{GSz10}, for the $m=2$ case. The idea is to first work with the
  operator
  \[ \begin{aligned}
    G_0 : (C^{4,\alpha}_\delta)^T \times \overline{\mathfrak{h}}_0
    \times\mathbf{R} &\to (C^{0,\alpha}_{\delta -4})^T \\
    (\phi,f,t) &\mapsto
    \mathcal{D}^*_{\omega_\epsilon}\mathcal{D}_{\omega_\epsilon} -
    \l(f) + t\chi,
  \end{aligned}\]
  where $\chi$ is the function from Proposition~\ref{prop:Bl0} and
  \[ \mathcal{D}^*_{\omega_\epsilon}\mathcal{D}_{\omega_\epsilon} = -
  L_{\omega_\epsilon}(\phi) +
		\frac{1}{2}\nabla\mathbf{s}(\omega_\epsilon)\cdot
		\nabla\phi.\]
  One can then use the inverses in Propositions~\ref{prop:Mp} and
  \ref{prop:Bl0} to construct an approximate right inverse for $G_0$,
  which in turn can be used to show that $G_0$ has a bounded right
  inverse. If $\psi\in (C^{0,\alpha}_{\delta -4})^T_0$, then we can
  use this to find $\phi\in
  (C^{4,\alpha}_\delta)^T,f\in\overline{\mathfrak{h}}_0,t\in\mathbf{R}$
  such that
  \[
  \mathcal{D}^*_{\omega_\epsilon}\mathcal{D}_{\omega_\epsilon}(\phi) -
  \l(f) + t\chi = \psi. \]
  Integrating this over $\Bl_pM$ we find that $t=0$. This shows that
  we have constructed an inverse for $G_1$. 
\end{proof}

\begin{rem}\label{rem:G}
 We will need analogous results for operators corresponding
  to a perturbation $\Omega = \omega_\epsilon + \ddbar \Phi$. We have
\[ \Vert (\mathcal{D}^*_{\omega_\epsilon}\mathcal{D}_{\omega_\epsilon}
-
\mathcal{D}^*_\Omega\mathcal{D}_\Omega)\phi\Vert_{C^{0,\alpha}_{\delta-4}}
\leqslant
C\Vert\Phi\Vert_{C^{4,\alpha}_2}\Vert\phi\Vert_{C^{4,\alpha}_\delta},\]
and
\[ \Vert \l(f) - \l_\Omega(f)\Vert_{C^{0,\alpha}_{\delta-4}} \leqslant
C\Vert \nabla \l(f)\cdot\nabla\Phi\Vert_{C^{0,\alpha}_{\delta-4}}
\leqslant C|f| \Vert\Phi\Vert_{C^{4,\alpha}_{\delta-3}} \leqslant
C|f|\Vert\Phi\Vert_{C^{4,\alpha}_2},\]
if $\delta - 3 < 2$. So as long as $\Vert\Phi\Vert_{C^{4,\alpha}_2}$
is sufficiently small, we can deduce the invertibility of the
operators corresponding to $\Omega$ from the invertibility of those
corresponding to $\omega_\epsilon$. Note also that
\begin{equation}\label{eq:intlOmega}
\begin{aligned}
\left|\int_{\Bl_pM} \l_\Omega(f)\,\Omega^m - \int_{\Bl_pM}
  \l(f)\,\omega_\epsilon^m\right| &\leqslant \int_{\Bl_pM}
\frac{1}{2}|\nabla\l(f)\cdot \nabla\Phi|\,\Omega^m \\ &\qquad + \left|
  \int_{\Bl_pM} \l(f)\,(\Omega^m - \omega_\epsilon^m)\right| \\
&\leqslant C |f|\,\Vert\Phi\Vert_{C^{4,\alpha}_2},
\end{aligned}
\end{equation}
so if $f\in\overline{\mathfrak{h}_0}$, then we can adjust $f$ while
preserving its norm up to a factor, to ensure that $\l_\Omega(f)$ has
zero integral with respect to $\Omega$, as long as
$\Vert\Phi\Vert_{C^{4,\alpha}_2}$ is sufficiently small. It follows
that both propositions can be applied to small perturbations of
$\omega_\epsilon$. 
\end{rem}

\subsection{The approximate solution $\Omega_1$}
\label{sec:approx}
We will now work on obtaining a metric $\Omega_1$ on $\Bl_pM$,
which is closer to being extremal than our previous candidate
$\omega_\epsilon$. In the next section we will use this to find an
even better approximate solution $\Omega_2$, at which point we will be
able to use the contraction mapping theorem to obtain a solution of
our equation. 

There are 3 regions in $\Bl_pM$ which we need to think
about differently, namely the region $\Bl_pM\setminus B_{2r_\epsilon}$,
the annular region $B_{2r_\epsilon}\setminus B_{r_\epsilon}$ on which
our cutoff fuctions $\gamma_1$ and $\gamma_2$ live, and
$B_{r_\epsilon}$. Here $r_\epsilon = \epsilon^\alpha$, and from now on
we will work with 
\[ \alpha = \frac{2m-1}{2m+1}. \]

Like in Section~\ref{sec:omegaepsilon}, 
let us write the metric $\omega$ in coordinates near the point $p$. We
will need a more precise expansion than before, so we write
\begin{equation}\label{eq:om1}
	\omega = \ddbar\left(\frac{|z|^2}{2} + A_4(z) + A_5(z) + \phi_6(z)\right),
\end{equation}
where $A_4(z), A_5(z)$ are quartic and quintic in $z$ respectively,
and $\phi_6\in C^{k,\alpha}_6(M_p)$. Also let us write $\mathbf{s} =
\mathbf{s}(\omega)$ for the scalar curvature of $\omega$. 
\begin{lem}\label{lem:A4A5}
Suppose that $\nabla\mathbf{s}$ vanishes at $p$. Then we have
\[ \begin{aligned}
  \Delta_0^2 A_4 &= -\mathbf{s}(p), \\
  \Delta_0^2 A_5 &= 0, 
\end{aligned}\]
where $\Delta_0$ is the Laplacian with respect to the Euclidean
metric. 
\end{lem}
\begin{proof}
This follows from computing the scalar curvature of $\omega$ as a
perturbation of the flat metric near $p$. 
\[ \mathbf{s}(\omega) = -\Delta_0^2(A_4 + A_5 + \phi_6) + Q_0(A_4 + A_5
+ \phi_6). \]
From Lemma~\ref{lem:Q} it follows that 
\[ Q_0(A_4 + A_5 + \phi_6) \in C^{k,\alpha}_2(B_1\setminus\{p\}), \]
and so
\[ \mathbf{s}(\omega) + \Delta_0^2(A_4 + A_5) \in C^{k,\alpha}_2, \]
near $p$. 
Since $\Delta_0^2 A_4$ is a constant and $\Delta_0^2 A_5$ is linear,
the result follows. 
\end{proof}

As for the rescaled Burns-Simanca metric, 
after a change of coordinates we can write it as
\begin{equation}\label{eq:eta1}
\begin{aligned}
	\epsilon^2\eta = \ddbar\Big(\frac{|z|^2}{2} &+ d_0\epsilon^{2m-2}
|z|^{4-2m} + d_1\epsilon^{2m}|z|^{2-2m} \\ &+
d_2\epsilon^{4m-4}|z|^{6-4m} +
\epsilon^2\psi_{4-4m}(\epsilon^{-1}z)\Big),
\end{aligned}
\end{equation}
where $\psi_{4-4m}\in C^{k,\alpha}_{4-4m}(\Bl_0\mathbf{C}^m)$, and $d_0
= -\frac{1}{2\pi^{m-1}(m-2)}$. 

These are the two metrics that we want to glue across the annular
region $B_{2r_\epsilon}\setminus B_{r_\epsilon}$. In the construction of
$\omega_\epsilon$ we performed this gluing by multiplying all the terms
except $|z|^2/2$ by cutoff functions. In order to get a better
approximate solution, we want to only multiply $\phi_6$ and
$\psi_{4-4m}$ by cutoff functions. For this we need to modify
$\epsilon^2\eta$ so that it contains $A_4(z), A_5(z)$, and we need to modify
$\omega$ by 
\[d_0\epsilon^{2m-2}|z|^{4-2m} + d_1\epsilon^{2m}|z|^{2-2m}
+ d_2\epsilon^{4m-4}|z|^{6-4m}.\]

Let us focus on $\omega$ first. For this we have the following.
\begin{lem}\label{lem:G}
We can find $T$-invariant functions $G_1, G_2$ on $M\setminus\{p\}$
such that distributionally on $M$ we have
\begin{equation}\label{eq:DDstarG} \begin{aligned}
           \mathcal{D}^*_\omega\mathcal{D}_\omega G_1 &= f_1 +
           \frac{4\pi^m}{(m-3)!} \delta_p \\
           \mathcal{D}^*_\omega\mathcal{D}_\omega G_2 &= f_2 -
           \frac{2\pi^m}{(m-2)!} \Delta\delta_p,
\end{aligned}\end{equation}
for some $f_1,f_2\in\overline{\mathfrak{h}}$, 
and for any $\delta > 0$ we have
\begin{equation}\label{eq:Gexp}
\begin{aligned}
       G_1 - |z|^{4-2m} &\in C^{k,\alpha}_{6-2m-\delta}(M_p), \\
       G_2 - |z|^{2-2m} &\in C^{k,\alpha}_{4-2m-\delta}(M_p).
\end{aligned}\end{equation}
In addition
\[ \begin{aligned}
     f_1 &= -\frac{4\pi^m}{(m-3)!}(V^{-1} + \mu(p)), \\
     f_2 &= \frac{2\pi^m}{(m-2)!} \Delta\mu(p),
\end{aligned} \]
where $V=\mathrm{Vol}(M)$. 
\end{lem}
\begin{proof}
Let us define $\widetilde{G}_1$ using a cutoff function
to be equal to zero on $M\setminus B_1$, and equal to $|z|^{4-2m}$ on
$B_{1/2}$. Comparing $\omega$ to the flat metric near $p$, we find
that
\[ \mathcal{D}^*_\omega\mathcal{D}_\omega \widetilde{G}_1 \in
C^{k-4,\alpha}_{2-2m}(M_p).\] 
We can now use the inverse in Proposition~\ref{prop:Mp}. For small
$\delta > 0$ we have $6-2m-\delta\in (4-2m,0)$ (for $m > 3$ we can let
$\delta=0$), so we obtain a $\phi\in C^{k,\alpha}_{6-2m-\delta}$ and
$f_1\in \overline{\mathfrak{h}}$ such that
\[ \mathcal{D}^*_\omega\mathcal{D}_\omega( \widetilde{G}_1 - \phi) =
f_1,\text{ on }M_p, \]
and so we can let $G_1 = \widetilde{G}_1 - \phi$. The only
contribution to the distributional
part of $\mathcal{D}^*_\omega\mathcal{D}_\omega \widetilde{G}$ comes
from $\Delta_0^2 |z|^{4-2m}$ in the flat metric, giving the result. 

Similarly we define $\widetilde{G}_2$ using a cutoff function to equal
$|z|^{2-2m}$ on $B_{1/2}$ and to vanish outside $B_1$. Comparing with
the flat metric again, we obtain
\[ \mathcal{D}^*_\omega\mathcal{D}_\omega\widetilde{G} \in
C^{k-4,\alpha}_{-2m}(M_p). \]
Once again we can find $\phi\in C^{k,\alpha}_{4-2m-\delta}$ for small
$\delta > 0$ (note that $4-2m$ is an indicial root), and $f_2\in
\overline{\mathfrak{h}}$ such that
\[ \mathcal{D}^*_\omega\mathcal{D}_\omega (\widetilde{G}_2 - \phi) =
f_2,\text{ on }M_p. \]
In this case, there are several contributions to the distributional
part at $p$, but apart from the leading contribution of
$\Delta\delta_p$, the rest is a multiple of $\delta_p$. We can
therefore find a constant $C$ such that $G_2 = \widetilde{G}_2 - \phi
+ CG_1$ satisfies our requirements. 

In order to find $f_1$, we take the $L^2$-product of
\eqref{eq:DDstarG} with any element $g\in\mathfrak{h}$, to obtain
\[ 0 = \langle f_1, g\rangle + \frac{4\pi^m}{(m-3)!}g(p), \]
since $\mathcal{D}^*_\omega\mathcal{D}_\omega g =0$. By definition
$g(p) = \langle\mu(p),g\rangle$, and so it follows that the projection
of $f_1$ onto $\mathfrak{h}$ must be 
\[ \mathrm{pr}_\mathfrak{h} f_1 = -\frac{4\pi^m}{(m-3)!}\mu(p). \]
To obtain $f_1$ from this, we just need to take the $L^2$-product of
\eqref{eq:DDstarG} with the function 1. 
We can obtain the formula for $f_2$ similarly. 
\end{proof}

\begin{lem}\label{lem:Gamma}
	We can find a function $\Gamma$ on $M\setminus B_{r_\epsilon}$ such that
	\begin{equation}\label{eq:DDGamma}
          \mathcal{D}_\omega^*\mathcal{D}_\omega\Gamma = -h_1 \text{ on
        }M\setminus B_{1}
        \end{equation}
	for some $h_1\in\overline{\mathfrak{h}}$, and satisfying the following
        properties. On $B_{2r_\epsilon}\setminus B_{r_\epsilon}$ the
        function $\Gamma$ has the form
	\[ \Gamma = d_0\epsilon^{2m-2}|z|^{4-2m} +
	d_1\epsilon^{2m}|z|^{2-2m} + d_2\epsilon^{4m-4}|z|^{6-4m} + \Gamma_1, \]
	where
	\[
        \Vert\Gamma_1\Vert_{C^{k,\alpha}_{3-2m}(B_{2r_\epsilon}\setminus
          B_{r_\epsilon})}= O(\epsilon^\kappa),\]
        for some $\kappa > 2m$.
        On $M\setminus B_{2r_\epsilon}$ we have
        \begin{equation}\label{eq:annulusest}
          \Big\Vert \mathbf{s}(\omega + \ddbar\Gamma) -
          \frac{1}{2}\nabla\mathbf{s}(\omega)\cdot \nabla\Gamma -
          \mathbf{s}(\omega) -
          h_1\Big\Vert_{C^{0,\alpha}_{-1-2m}(M\setminus B_{2\epsilon})} = O(\epsilon^{\kappa}),
          \end{equation}
          for some $\kappa > 2m$. In addition
          \begin{equation}\label{eq:h1eq}\begin{aligned}
            h_1 &= -\frac{2\pi \epsilon^{2m-2}}{(m-2)!}(V^{-1} +
          \mu(p)) + \frac{\epsilon^{2m} \mathbf{s}(p)}{m!}(V^{-1} +
          \mu(p)) - d_1 \epsilon^{2m}
          \frac{2\pi^m}{(m-2)!}\Delta\mu(p) \\
          &= C - \epsilon^{2m-2}\left(c_1 -
            \frac{\epsilon^2}{m!}\mathbf{s}(p)\right)\mu(p) - 
          \epsilon^{2m}c_2\Delta\mu(p) ,
          \end{aligned}\end{equation}
          where $V = \mathrm{Vol}(M)$, $C$ is a constant, and
          $c_1,c_2 > 0$.  
\end{lem}
\begin{proof}
Let us use a cutoff function to define $G_3$ to be zero outside $B_1$,
and equal to $|z|^{6-4m}$ in $B_{1/2}$. We let
\[ \Gamma = d_0\epsilon^{2m-2}G_1 + d_1 \epsilon^{2m} G_2 -
\epsilon^{2m}\frac{(m-3)!\mathbf{s}(p)}{4\pi^mm!} G_1 + 
d_2\epsilon^{4m-4}G_3, \]
where $G_1$ and $G_2$ are defined in Lemma~\ref{lem:G}. Then
\eqref{eq:DDGamma} and \eqref{eq:h1eq} follow 
from the properties of $G_1,G_2$.   
On the annulus $B_{2r_\epsilon}\setminus
B_{r_\epsilon}$
 we have
\[ \Gamma =  d_0\epsilon^{2m-2}|z|^{4-2m} +
	d_1\epsilon^{2m}|z|^{2-2m} + d_2\epsilon^{4m-4}|z|^{6-4m} +
        \Gamma_1, \]
where 
\[ \Gamma_1 = d_0\epsilon^{2m-2}(G_1 - |z|^{4-2m}) +
d_1\epsilon^{2m}(G_2 - |z|^{2-2m}) + C\epsilon^{2m}|z|^{4-2m}, \]
for some constant $C$. 
From \eqref{eq:Gexp} it follows that
\[ \Vert\Gamma_1\Vert_{C^{k,\alpha}_{3-2m}(B_{2r_\epsilon}\setminus
  B_{r_\epsilon})} = O(\epsilon^\kappa) \]
for some $\kappa > 2m$. For instance
\[ \Vert G_1 -
|z|^{4-2m}\Vert_{C^{k,\alpha}_{3-2m}(B_{2r_\epsilon} \setminus
  B_{r_\epsilon})} \leqslant Cr_\epsilon^{3-\delta}\Vert G_1 -
|z|^{4-2m}\Vert_{ C^{k,\alpha}_{6-2m-\delta}(B_{2r_\epsilon}\setminus
  B_{r_\epsilon})}, \]
and for sufficiently small $\delta > 0$ we have 
\[ \epsilon^{2m-2}r_\epsilon^{3-\delta} = O(\epsilon^\kappa), \]
for some $\kappa > 2m$, by our choice of $r_\epsilon$, since
\[ \frac{2m-1}{2m+1} > \frac{2}{3}\]
for $m\geqslant 3$. The other terms are larger and are handled similarly. 

For the scalar curvature of $\omega + \ddbar\Gamma$ we have
\[ \mathbf{s}(\omega + \ddbar\Gamma) -
\frac{1}{2}\nabla\mathbf{s}(\omega)\cdot\nabla \Gamma  -
\mathbf{s}(\omega) - h_1 = Q_\omega(\Gamma) -
\mathcal{D}^*_\omega \mathcal{D}_\omega\Gamma - h_1. \]
We will work on the 3 regions $M\setminus B_1$, $B_1\setminus B_{1/2}$
and $B_{1/2}\setminus B_{2r_\epsilon}$ separately. 

On $M\setminus B_1$
we have $\mathcal{D}^*_\omega\mathcal{D}_\omega\Gamma + h_1 = 0$, and
\[ \Vert Q_\omega(\Gamma)\Vert_{C^{0,\alpha}_{-1-2m}(M\setminus B_1)}
\leqslant C\Vert \Gamma\Vert_{C^{4,\alpha}_2}\Vert
\Gamma\Vert_{C^{4,\alpha}_{3-2m}}. \]
The weight is irrelevant outside $B_1$, so we have
\[ \Vert Q_\omega(\Gamma)\Vert_{C^{0,\alpha}(M\setminus B_1)}
\leqslant C(\epsilon^{2m-2})^2 = O(\epsilon^\kappa) \]
for some $\kappa > 2m$, as long as $m > 2$. 

On $B_1\setminus B_{1/2}$ we can still ignore the weights, so
we still have the same estimate for
$Q_\omega(\Gamma)$, but now
\[ \mathcal{D}^*_\omega\mathcal{D}_\omega\Gamma + h_1 =
C\epsilon^{4m-4}\mathcal{D}^*_\omega\mathcal{D}_\omega(G_3).\]
It follows from this that
\[ \Vert \mathcal{D}^*_\omega\mathcal{D}_\omega\Gamma +
h_1\Vert_{C^{0,\alpha}( B_1\setminus B_{1/2})} = O(\epsilon^{4m-4}) =
O(\epsilon^\kappa) \]
for $\kappa > 2m$, as long as $m > 2$. 

The most delicate estimate is on $B_{1/2}\setminus B_{r_\epsilon}$. It
is best to work on the annuli $A_r = B_{2r}\setminus B_r$, for
$r\in(r_\epsilon, 1/4)$. We
have
\[ \mathcal{D}^*_\omega\mathcal{D}_\omega\Gamma + h_1 =
d_2\epsilon^{4m-4}\mathcal{D}^*_\omega\mathcal{D}_\omega
|z|^{6-4m}, \]
so
\[
 \Vert \mathcal{D}^*_\omega\mathcal{D}_\omega \Gamma + h_1 -
d_2\epsilon^{4m-4} \Delta^2_0 |z|^{6-4m}\Vert_{ C^{0,\alpha}_{-1-2m}} = d_2\epsilon^{4m-4}\Vert
(\mathcal{D}^*_\omega\mathcal{D}_\omega - \Delta^2_0)|z|^{6-4m}\Vert_{C^{0,\alpha}_{-1-2m}}.\]
On the annulus $A_r$ we have
\[\begin{aligned}
 \Vert (\mathcal{D}^*_\omega\mathcal{D}_\omega - \Delta^2_0)
|z|^{6-4m}\Vert_{C^{0,\alpha}_{-1-2m}} &\leqslant \Vert
\phi\Vert_{C^{4,\alpha}_2} \Vert |z|^{6-4m}\Vert_{C^{4,\alpha}_{3-2m}}
\\
&\leqslant C r^2\cdot r^{3-2m},
\end{aligned}\]
where $\phi=O(|z|^4)$. We have $r\geqslant r_\epsilon$, and so
\[ \epsilon^{4m-4}r^{5-2m} = O(\epsilon^\kappa) \]
for some $\kappa > 2m$, since $\alpha < 1$. 

We also have
\[ Q_\omega(\Gamma) = Q_\omega(d_0\epsilon^{2m-2}|z|^{4-2m}) + \Big[
Q_\omega(\Gamma) - Q_\omega(d_0\epsilon^{2m-2}|z|^{4-2m})\Big]. \]
The next highest order term in $\Gamma$ after
$d_0\epsilon^{2m-2}|z|^{4-2m}$ is $d_1\epsilon^{2m}|z|^{2-2m}$, so on
the annulus $A_r$
\[ \begin{aligned}
 \Vert  Q_\omega(\Gamma) - Q_\omega(d_0\epsilon^{2m-2}|z|^{4-2m})
\Vert_{C^{0,\alpha}_{-1-2m}} &\leqslant C\epsilon^{4m} \Vert
|z|^{2-2m}\Vert_{C^{4,\alpha}_2} \Vert |z|^{2-2m}\Vert_{C^{4,\alpha}_{
    3-2m}} \\
&\leqslant C\epsilon^{4m}r^{-2m}r^{-1} = O(\epsilon^\kappa),
\end{aligned}\]
for $\kappa > 2m$, since 
\[ \alpha < \frac{2m}{2m+1}. \]
What remains is to estimate
\[ Q_\omega(d_0\epsilon^{2m-2}|z|^{4-2m}) -
\Delta^2_0(d_2\epsilon^{4m-4}|z|^{6-4m}). \]   
It is not hard to check that both terms are  of the same order, and for
$m > 3$ they are sufficiently small. However for $m=3$ we need to work
harder. On the annulus $A_r$, using Lemma~\ref{lem:Q2}, we have
\[ \begin{aligned}
 \Vert Q_\omega(d_0\epsilon^{2m-2}|z|^{4-2m}) &-
Q_0(d_0\epsilon^{2m-2}|z|^{4-2m})\Vert_{C^{0,\alpha}_{-1-2m}}
\leqslant \\
&\qquad \leqslant Cr^2 \epsilon^{4m-4}\Vert |z|^{4-2m}\Vert_{C^{4,\alpha}_2}
\Vert |z|^{4-2m}\Vert_{C^{4,\alpha}_{3-2m}} \\
&\qquad \leqslant C\epsilon^{4m-4}r^{5-2m} = O(\epsilon^\kappa),
\end{aligned}\]  
for $\kappa > 2m$, where $Q_0$ is given by the flat metric. 
Using that $\epsilon^2\eta$ is scalar flat, we have
\[ 0 = \mathbf{s}\Big[\ddbar\Big(\frac{|z|^2}{2} +
d_0\epsilon^{2m-2}|z|^{4-2m} + d_1\epsilon^{2m}|z|^{2-2m} +
d_2\epsilon^{4m-4}|z|^{6-4m} +
\epsilon^2\psi_{4-4m}(\epsilon^{-1}z)\Big)\Big],\] 
so if we write $\epsilon^2\eta = \ddbar\Big(\frac{|z|^2}{2} +
\epsilon^2\psi(\epsilon^{-1}z)\Big)$, then we get
\[ \begin{aligned} -\Delta_0^2(d_2\epsilon^{4m-4}|z|^{6-4m}) &+
  Q_0(d_0\epsilon^{2m-2}|z|^{4-2m}) =
  \Delta_0^2(\epsilon^2\psi_{4-4m}(\epsilon^{-1}z)) + \\
&\quad + Q_0(d_0\epsilon^{2m-2}|z|^{4-4m}) -
Q_0(\epsilon^2\psi(\epsilon^{-1}z)),
\end{aligned}\]
and so on $A_r$ we have
\[\begin{aligned} \Vert -\Delta_0^2(d_2\epsilon^{4m-4}|z|^{6-4m}) &+
  Q_0(d_0\epsilon^{2m-2}|z|^{4-2m}) \Vert_{C^{0,\alpha}_{-1-2m}}
  \leqslant C\epsilon^{4m-2}\Vert
  |z|^{4-4m}\Vert_{C^{4,\alpha}_{3-2m}} + \\ &\qquad + C\epsilon^{4m}\Vert
  |z|^{2-2m}\Vert_{C^{4,\alpha}_2} \Vert
  |z|^{2-2m}\Vert_{C^{4,\alpha}_{3-2m}} \\
&\quad\leqslant C\epsilon^{4m-2}r^{1-2m} + C\epsilon^{4m} r^{-1-2m} = O(\epsilon^\kappa),  
\end{aligned} \]
for $\kappa > 2m$, where we used that the largest order term in
$\epsilon^2\psi(\epsilon^{-1}z)$ after the leading term is
$\epsilon^{2m}|z|^{2-2m}$. Combining all these estimates, we obtain
the required bound \eqref{eq:annulusest}. 
\end{proof}

Now to deal with modifying $\epsilon^2\eta$, we have the following. 
\begin{lem}\label{lem:Psi}
We can find a function $\Psi$ on $\Bl_0\mathbf{C}^m$ of the form
\[ \Psi = A_4(z) + A_5(z) + \Psi_1, \]
where 
\[ \Vert \Psi_1\Vert_{C^{k,\alpha}_{3-2m}(B_{2r_\epsilon}\setminus
  B_{r_\epsilon})} = O(\epsilon^\kappa), \]
and in addition
\[ \Vert \mathbf{s}(\epsilon^2\eta + \ddbar\Psi) -
\mathbf{s}(p)\Vert_{C^{0,\alpha}_{-1-2m}(B_{r_\epsilon})} =
O(\epsilon^\kappa), \]
for some $\kappa > 2m$. 
\end{lem}
\begin{proof}
We will work in terms of $\eta$ with the variable
$w=\epsilon^{-1}z$. Write $R_\epsilon = \epsilon^{-1}r_\epsilon$, so
that in terms of $w$, 
we are gluing on the annulus $B_{2R_\epsilon}\setminus
B_{R_\epsilon}$.

Write $\widetilde{A}_4(w), \widetilde{A}_5(w)$ for the functions
$A_4(w)$ and $A_5(w)$ cut off on the annulus $B_{4R_\epsilon}\setminus
B_{2R_\epsilon}$. Since
\[ \Vert \epsilon^2\widetilde{A}_4 +
\epsilon^3\widetilde{A}_5\Vert_{C^{4,\alpha}_2} \leqslant
C\epsilon^2R_\epsilon^2 \ll 1, \]
we have
\[\Vert Q_\eta(\epsilon^2\widetilde{A}_4 + \epsilon^3 \widetilde{A}_5)
\Vert_{C^{0,\alpha}_0} \leqslant C
(\epsilon^2R_\epsilon^2)(\epsilon^2) = C\epsilon^4R_\epsilon^2. \]
It follows from Proposition~\ref{prop:Bl0} that we can find
$\widetilde{\Psi}$ such that
\[ \mathcal{D}^*_\eta\mathcal{D}_\eta\widetilde{\Psi} =
Q_\eta(\epsilon^2\widetilde{A}_4 + \epsilon^3\widetilde{A}_5),\]
and
\[ \Vert \widetilde{\Psi}\Vert_{C^{4,\alpha}_4} \leqslant
C\epsilon^4R_\epsilon^2 = C\epsilon^2r_\epsilon^2. \]
Setting $\Psi_1(z) = \epsilon^2\widetilde{\Psi}(\epsilon^{-1}z)$,
we then have
\[ \Vert \Psi_1\Vert_{C^{4,\alpha}_{3-2m}(B_{2r_\epsilon}\setminus
  B_{r_\epsilon})} \leqslant Cr_\epsilon^{2m+1}
\Vert\Psi_1\Vert_{C^{4,\alpha}_4(B_{2r_\epsilon}\setminus
  B_{r_\epsilon})} \leqslant Cr_\epsilon^{2m+1}r_\epsilon^2 =
O(\epsilon^\kappa), \]
for $\kappa > 2m$, since
\[ \alpha > \frac{2m}{2m+3}. \]
In addition, using Lemma~\ref{lem:A4A5},
\[\begin{aligned} \mathbf{s}(\eta + \ddbar(\epsilon^2\widetilde{A}_4 +
\epsilon^3\widetilde{A}_5 &+ \widetilde{\Psi})) =
L_\eta(\epsilon^2\widetilde{A}_4 + \epsilon^3\widetilde{A}_5 +
\widetilde{\Psi}) + Q_\eta(\epsilon^2\widetilde{A}_4 +
\epsilon^3\widetilde{A}_5 + \widetilde{\Psi}) \\
&= \epsilon^2\mathbf{s}(p) + L_\eta(\widetilde{\Psi}) + (L_\eta +
\Delta^2_0)(\epsilon^2\widetilde{A}_4 + \epsilon^3\widetilde{A}_5) \\
&\quad + Q_\eta(\epsilon^2\widetilde{A}_4 +
\epsilon^3\widetilde{A}_5) \\
&\quad + Q_\eta(\epsilon^2\widetilde{A}_4 + \epsilon^3\widetilde{A}_5
+ \widetilde{\Psi}) - Q_\eta(\epsilon^2\widetilde{A}_4 +
\epsilon^4\widetilde{A}_5). 
\end{aligned} \]
Using the equality $\mathcal{D}^*_\eta\mathcal{D}_\eta = L_\eta$ we
have $L_\eta(\widetilde{\Psi}) + Q_\eta(\epsilon^2\widetilde{A}_4 +
\epsilon^3\widetilde{A}_5) = 0$. Using
the fact that $\eta$ differs from the flat metric by order
$|w|^{2-2m}$, we get
\[ \Vert \mathbf{s}(\eta + \ddbar(\epsilon^2\widetilde{A}_4 +
\epsilon^3\widetilde{A}_5 + \widetilde{\Psi})) -
\epsilon^2\mathbf{s}(p)\Vert_{C^{ 0,\alpha}_{-1-2m}(B_{2R_\epsilon})} \leqslant
C(\epsilon^2R_\epsilon^3 + \epsilon^8R_\epsilon^{7+2m}) =
O(\epsilon^{1+\delta}), \]
for some $\delta > 0$. 
Since  on the ball $B_{R_\epsilon}$ in terms of $w$ (and on
$B_{r_\epsilon}$ in terms of $z$) we have
\[ \mathbf{s}(\epsilon^2\eta + \ddbar\Psi) =
\epsilon^{-2}\mathbf{s}(\eta + \ddbar(\epsilon^2\widetilde{A}_4 +
\epsilon^3\widetilde{A}_5 + \widetilde{\Psi})), \]
it follows that
\[ \Vert \mathbf{s}(\epsilon^2\eta + \ddbar\Psi) -
\mathbf{s}(p)\Vert_{C^{0,\alpha}_{-1-2m}(B_{r_\epsilon})} = O(\epsilon^\kappa), \]
for some $\kappa > 2m$. 
\end{proof}

We now define our new approximate metric $\Omega_1$ to be equal
to $\omega + \ddbar\Gamma$ on $M\setminus B_{2r_\epsilon}$, equal to
$\epsilon^2\eta + \ddbar\Psi$ on $B_{r_\epsilon}$, and on the annular
region $B_{2r_\epsilon}\setminus B_{r_\epsilon}$ we let
\[ \begin{aligned}
 \Omega_1 =& \ddbar\Big(\frac{|z|^2}{2} + A_4 + A_5 +
 \gamma_1\phi_6 + \gamma_1\Gamma_1 +
  d_0\epsilon^{2m-2}|z|^{4-2m} + d_1\epsilon^{2m}|z|^{2-2m} \\ &+
  d_2\epsilon^{4m-4}|z|^{6-4m} + \gamma_2
  \epsilon^2\psi_{4-4m}(\epsilon^{-1}z) + \gamma_2\Psi_1\Big).
\end{aligned}\]

\begin{lem}\label{lem:annularscal}
On the annular region $B_{2r_\epsilon}\setminus B_{r_\epsilon}$ we
have
\[ \Vert \mathbf{s}(\Omega_1) - \mathbf{s}(p)
\Vert_{C^{0,\alpha}_{-1-2m}(B_{2r_\epsilon}\setminus B_{r_\epsilon})}=
O(\epsilon^\kappa), \]
for some $\kappa > 2m$. 
\end{lem}
\begin{proof}
We compute the scalar curvature of $\Omega_1$ as a
perturbation of the metric
\[ \omega_0 = \ddbar\Big(\frac{|z|^2}{2} + 
d_0\epsilon^{2m-2}|z|^{4-2m} + d_1\epsilon^{2m}|z|^{2-2m} +
  d_2\epsilon^{4m-4}|z|^{6-4m}\Big)\]
on the annulus $B_{2r_\epsilon}\setminus B_{r_\epsilon}$. Since 
\[ \mathbf{s}(\omega_0 +
\epsilon^2\ddbar\psi_{4-4m}(\epsilon^{-1}z))=\mathbf{s}(\epsilon^2\eta)=0, \]
we have
\[ \begin{aligned}
\mathbf{s}(\omega_0) &= -\epsilon^2
L_{\omega_0}(\psi_{4-4m}(\epsilon^{-1}z)) -
Q_{\omega_0}(\epsilon^2\psi_{4-4m} (\epsilon^{-1}z)) \\
&= \epsilon^2\Delta_0^2(\psi_{4-4m}(\epsilon^{-1}z)) + \text{lower
  order terms}. 
\end{aligned}\]
It follows that 
\[ \Vert \mathbf{s}(\omega_0)
\Vert_{C^{0,\alpha}_{-1-2m}(B_{2r_\epsilon}\setminus B_{r_\epsilon})}
= O(\epsilon^{4m-2}r_\epsilon^{1-2m}) = O(\epsilon^\kappa) \]
for some $\kappa > 2m$.  Then
\[ \begin{aligned}
\mathbf{s}(\Omega_1) &= \mathbf{s}(\omega_0) +
L_{\omega_0}(A_4 + A_5 + \gamma_1\phi_6 + \gamma_1\Gamma_1 +
\gamma_2\epsilon^2\psi_{4-4m}(\epsilon^{-1}z) + \gamma_2\Psi_1) \\
&\quad + Q_{\omega_0}(A_4 + A_5 + \gamma_1\phi_6 + \gamma_1\Gamma_1 +
\gamma_2\epsilon^2\psi_{4-4m}(\epsilon^{-1}z) + \gamma_2\Psi_1).
\end{aligned}\]
On the annulus $B_{2r_\epsilon}\setminus B_{r_\epsilon}$ we have
\[ A_4 + A_5 + \gamma_1\phi_6 + \gamma_1\Gamma_1 +
\gamma_2\epsilon^2\psi_{4-4m}(\epsilon^{-1}z) + \gamma_2\Psi_1 = A_4 +
\text{lower order terms}, \]
and also
\[ \begin{aligned}
\Vert \gamma_1\phi_6 + \gamma_1\Gamma_1 +
\gamma_2\epsilon^2\psi_{4-4m}(\epsilon^{-1}z) +
\gamma_2\Psi_1\Vert_{C^{4,\alpha}_{3-2m}} &= O(r_\epsilon^{2m+3} +
\epsilon^{4m-2}r_\epsilon^{1-2m} + \epsilon^\kappa)\\
&= O(\epsilon^\kappa),
\end{aligned}\]
for some $\kappa > 2m$,
and
\[ \Delta_0^2(A_4 + A_5) = -\mathbf{s}(p). \] 
It follows that 
\[\begin{aligned} \Vert\mathbf{s}(\Omega_1) -
\mathbf{s}(p)\Vert_{C^{0,\alpha}_{-1-2m}} &\leqslant \Vert
(L_{\omega_0} + \Delta^2_0)(A_4+A_5)\Vert_{C^{0,\alpha}_{-1-2m}} +
O(\epsilon^\kappa) \\ &\quad + C\Vert A_4\Vert_{C^{4,\alpha}_2}\Vert
A_4\Vert_{C^{4,\alpha}_{3-2m}} \\
&= O(\epsilon^\kappa),
\end{aligned} \] 
for some $\kappa > 2m$. 

\end{proof}

Let us write
$\Omega_1 = \omega_\epsilon + \ddbar u_1$.
From the results above, we find
\begin{lem}\label{lem:tildeomega}
We have
\begin{equation}\label{eq:Omega1est}
\Vert \mathbf{s}(\Omega_1) - \l_{\Omega_1}(\mathbf{s} +
h_1)\Vert_{C^{0,\alpha}_{-1-2m}} =
O(\epsilon^\kappa), 
\end{equation}
for some $\kappa > 2m$, and
\begin{equation}\label{eq:u1est}
 \Vert u_1\Vert_{C^{k,\alpha}_3} = O(\epsilon^\delta),
\end{equation}
for some $\delta > 0$. 
\end{lem}
\begin{proof}
We work on the regions $M\setminus B_{2r_\epsilon}$ and $B_{2\epsilon}$ 
separately. On $M\setminus B_{2r_\epsilon}$ we have $u_1
= \Gamma$, $\omega_\epsilon = \omega$ and $\l(h_1) = h_1$, so the result
follows from Lemma~\ref{lem:Gamma}, together with 
\[ \Vert \l_{\Omega_1}(h_1) - h_1\Vert_{C^{0,\alpha}_{-1-2m}} =
\left\Vert \frac{1}{2}\nabla\Gamma\cdot\nabla h_1
\right\Vert_{C^{0,\alpha}_{-1-2m}} \leqslant C\epsilon^{4m-4} =
O(\epsilon^\kappa), \]
for $\kappa > 2m$, since $m>2$. 

On $B_{2r_\epsilon}$ we have
\[ \Vert \l_{\Omega_1}(h_1) \Vert_{C^{0,\alpha}_{-1-2m}} \leqslant
Cr_\epsilon^{2m+1} \Vert \l_{\Omega_1}(h_1)\Vert_{C^{0,\alpha}_0}
\leqslant Cr_\epsilon^{2m+1} \epsilon^{2m-2} = O(\epsilon^\kappa)\]
for some $\kappa > 2m$, since $\alpha > \frac{2}{2m+1}$. In addition
\[ \Vert \l_{\Omega_1}(\mathbf{s} -
\mathbf{s}(p))\Vert_{C^{0,\alpha}_{-1-2m}} \leqslant Cr_\epsilon^{2m+3}\Vert
\l_{\Omega_1}(\mathbf{s}- \mathbf{s}(p))\Vert_{C^{0,\alpha}_2}
\leqslant Cr_\epsilon^{2m+3} = O(\epsilon^\kappa), \]
for some $\kappa > 2m$. The bound \eqref{eq:Omega1est} then follows from Lemmas
\ref{lem:Psi} and \ref{lem:annularscal}.

As for \eqref{eq:u1est}, note that outside $B_{2r_\epsilon}$ the
leading term in $u_1$ is of order $\epsilon^{2m-2}|z|^{4-2m}$,
while inside $B_{2r_\epsilon}$ the leading term is of order
$|z|^4$. The bound \eqref{eq:u1est} is then easy to check. 
\end{proof}

\subsection{The approximate solution $\Omega_2$}
We need to modify $\Omega_1$ once more to
obtain a metric $\Omega_2 = \Omega_1 + \ddbar u_2$. For this we
want to find $(u_2, h_2)$ solving
\begin{equation}\label{eq:u2}
-\mathcal{D}^*_{\Omega_1}\mathcal{D}_{\Omega_1} u_2 -
\l_{\Omega_1}(h_2) = \l_{\Omega_1}(\mathbf{s} + h_1) -
\mathbf{s}(\Omega_1). 
\end{equation}
We will use the inverse operator from Proposition~\ref{prop:inverse2},
so in addition to the estimate from Lemma~\ref{lem:tildeomega}, we
need to bound the integral of $\l_{\Omega_1}(\mathbf{s} + h_1) -
\mathbf{s}(\Omega_1)$. 
For this we have the following.
\begin{lem} We have
\[\int_{\Bl_pM} \l_{\Omega_1}(\mathbf{s} + h_1) -
\mathbf{s}(\Omega_1)\,
\frac{\Omega_1^m}{m!} = 
O(\epsilon^\kappa), \]
for some $\kappa > 2m$. 
\end{lem}
\begin{proof}
First note that from Lemma~\ref{lem:intl} we have
\[ \int_{\Bl_pM} \l_{\Omega_1}(\mathbf{s}) -
\mathbf{s}(\Omega_1)\,
\frac{\Omega_1^m}{m!} = 
\int_{\Bl_pM} \l(\mathbf{s}) - \mathbf{s}(\omega_\epsilon)\,
\frac{\omega_\epsilon^m}{m!}, \]
since $\mathbf{s}\in\overline{\mathfrak{t}}$ and the total scalar
curvature is an invariant of the K\"ahler class. 

We have
\begin{equation}\label{eq:inth1Omega1}
  \begin{aligned}
 \int_{\Bl_pM} \l_{\Omega_1}(h_1)\,\Omega_1^m - \int_{\Bl_pM}
\l(h_1)\,\omega_\epsilon^m &= \int_{\Bl_pM}
\frac{1}{2}\nabla\l(h_1)\cdot\nabla u_1\,\Omega_1^m + \\
&\qquad+ \int_{\Bl_pM}
\l(h_1)\,(\Omega_1^m - \omega_\epsilon^m).
\end{aligned}
\end{equation}
On $M\setminus B_1$, we can bound this by
$|h_1|\,\Vert u_1\Vert_{C^{4,\alpha}_2}$, as in
Equation~\eqref{eq:intlOmega}.  We have
\[ |h_1|\,\Vert u_1\Vert_{C^{4,\alpha}(M\setminus B_1)}\leqslant
C\epsilon^{2m-2}\epsilon^{2m-2} = O(\epsilon^\kappa), \]
for some $\kappa > 2m$ if $m > 2$. To bound \eqref{eq:inth1Omega1} on
$B_1\setminus B_{2r_\epsilon}$, note that
\[ \Vert u_1\Vert_{C^{4,\alpha}_{4-2m}(M\setminus B_{2r_\epsilon})} =
O(\epsilon^{2m-2}), \]
and $|\l(h_1)| + |\nabla\l(h_1)| \leqslant C\epsilon^{2m-2}$ from
Lemma~\ref{lem:lb}.
It follows that
\[ \Vert\l(h_1)\,
(\Omega_1^m-\omega_\epsilon^m)\Vert_{C^0_{4-2m}(B_1\setminus
  B_{2r_\epsilon})} \leqslant C\epsilon^{4m-4}. \]
The integral on $B_1\setminus B_{2r_\epsilon}$ is then
bounded by
\[ \int_{2r_\epsilon}^1 \epsilon^{4m-4}r^{2-2m}r^{2m-1}\,dr \leqslant
C\epsilon^{4m-4}.\]
Similarly
\[ \Vert \nabla\l(h_1)\cdot\nabla u_1\Vert_{C^0_{3-2m}}\leqslant
C\epsilon^{4m-4}, \]
so the other term in \eqref{eq:inth1Omega1} is also bounded by
$\epsilon^{4m-4}$. Since $m > 2$, we have $\epsilon^{4m-4} =
O(\epsilon^\kappa)$ for some $\kappa > 2m$.

The lift $\l(h_1)$ equals $h_1$ outside $B_{2r_\epsilon}$, and the
volume of $B_{2r_\epsilon}$ is of order $r_\epsilon^{2m}$, so we have
\[ \int_{\Bl_pM} \l(h_1)\,\frac{\omega_\epsilon^m}{m!} = \int_{M}
h_1\,\frac{\omega^m}{m!} + O(\epsilon^{2m-2}r_\epsilon^{2m}) = \int_M
h_1\,\frac{\omega^m}{m!} + O(\epsilon^\kappa),\]
for some $\kappa > 2m$.

Since $\mathbf{s}\in\mathfrak{t}$, we can relate the integrals on
$\Bl_pM$ and on $M$ using the formulas in
Proposition~\ref{prop:Blformulas}. We have
\[ 
\int_{\Bl_pM} \l(\mathbf{s})\,\frac{\omega_\epsilon^m}{m!} =
\int_M \mathbf{s}\, \frac{\omega^m}{m!} -
\frac{\epsilon^{2m}}{m!}\mathbf{s}(p) + O(\epsilon^\kappa),
 \]
for some $\kappa > 2m$, and also
\[ \int_{\Bl_pM}
\mathbf{s}(\omega_\epsilon)\,\frac{\omega_\epsilon^m}{ m!} = \int_M
\mathbf{s}\,\frac{\omega^m}{m!} - \frac{2\pi\epsilon^{2m-2}}{(m-2)!}. \]

Combining these formulas, we have
\[ \int_{\Bl_pM} \l(\mathbf{s} + h_1) - \mathbf{s}(\omega_\epsilon)\,
\frac{\omega_\epsilon^m}{m!} = \int_M h_1\,\frac{\omega^m}{m!} +
\frac{2\pi\epsilon^{2m-2}}{(m-2)!} - \frac{\epsilon^{2m}}{m!}\mathbf{s}(p) +
O(\epsilon^\kappa). \]

From the formula for $h_1$ in Lemma~\ref{lem:Gamma}, the first 3 terms
on the right cancel, leaving only $O(\epsilon^\kappa)$. 
\end{proof}

Letting $C$ be the average of $\l_{\Omega_1}(\mathbf{s}+h_1) -
\mathbf{s}(\Omega_1)$ with respect to $\Omega_1$, we can apply
Proposition~\ref{prop:inverse2} to find $u_2$ and
$h_2'\in\overline{\mathfrak{h}}_0$ such that
\[ -\mathcal{D}^*_{\Omega_1}\mathcal{D}_{\Omega_1} u_2 -
\l_{\Omega_1}(h_2') = \l_{\Omega_1}(\mathbf{s} + h_1) -
\mathbf{s}(\Omega_1) - C.\]
Letting $h_2 = h_2' - C$, this implies
that we can solve \eqref{eq:u2}, with
\begin{equation}\label{eq:u2h2}
 \Vert u_2\Vert_{C^{4,\alpha}_{3-2m}} + |h_2| =
O(\epsilon^\kappa), 
\end{equation}
with $\kappa > 2m$. 
We now let $\Omega_2 = \Omega_1 + \ddbar u_2$. This satisfies the
following.
\begin{lem}\label{lem:Omega2}
\[\Vert \mathbf{s}(\Omega_2) - \l_{\Omega_2}(\mathbf{s} +
h_1+h_2)\Vert_{C^{0,\alpha}_{-2m}} = O(\epsilon^\kappa), \]
for some $\kappa > 2m$. 
\end{lem}
\begin{proof}
First we have
\[ \begin{aligned}
\mathbf{s}(\Omega_2) &= \mathbf{s}(\Omega_1) + L_{\Omega_1}(u_2) +
Q_{\Omega_1}(u_2) \\ 
&= \mathbf{s}(\Omega_1) -
\mathcal{D}^*_{\Omega_1}\mathcal{D}_{\Omega_1}(u_2) +
\frac{1}{2}\nabla\mathbf{s}(\Omega_1)\cdot_{\Omega_1} \nabla u_2 +
Q_{\Omega_1}(u_2) \\
&= \l_{\Omega_1}(\mathbf{s} + h_1 + h_2) + \frac{1}{2}
\nabla\mathbf{s}(\Omega_1)\cdot_{\Omega_1} \nabla u_2 + 
Q_{\Omega_1}(u_2) \\
&= \l_{\Omega_2}(\mathbf{s}+h_1+h_2) - \frac{1}{2}\nabla\l(\mathbf{s}
+ h_1+h_2)\cdot\nabla u_2 \\
&\qquad + \frac{1}{2}
\nabla\mathbf{s}(\Omega_1)\cdot_{\Omega_1} \nabla u_2 + 
Q_{\Omega_1}(u_2),
\end{aligned}\]
where by $\cdot_{\Omega_1}$ we indicate that the gradients and inner
products are taken with respect to $\Omega_1$ instead of
$\omega_\epsilon$.  We need to estimate $\nabla\mathbf{s}(\Omega_1)$,
and for this we have
\[ \Vert \mathbf{s}(\Omega_1) -
\mathbf{s}(\omega_\epsilon)\Vert_{C^{1,\alpha}_{-1}} = \Vert
L_{\omega'}(u_1)\Vert_{C^{1,\alpha}_{-1}} \leqslant C\Vert
u_1\Vert_{C^{5,\alpha}_3} \leqslant C, \]
where $\omega' = \omega_\epsilon + t\ddbar u_1$ for some
$t\in[0,1]$, where we used \eqref{eq:u1est}. 
We then have
\[ \Vert \nabla\l(\mathbf{s}+h_1+h_2)\cdot\nabla u_2 -
\nabla\mathbf{s}(\Omega_1)\cdot_{\Omega_1} \nabla
u_2\Vert_{C^{0,\alpha}_{-2m}} \leqslant C\Vert
u_2\Vert_{C^{4,\alpha}_{2-2m}} = O(\epsilon^\kappa), \]
by \eqref{eq:u2h2}. So we only need to estimate $Q_{\Omega_1}(u_2)$,
but for this Lemma~\ref{lem:Q} implies
\[
\begin{aligned}
  \Vert Q_{\Omega_1}(u_2)\Vert_{C^{0,\alpha}_{-2m}} &\leqslant C\Vert
u_2\Vert_{C^{4,\alpha}_2}\,\Vert u_2\Vert_{C^{4,\alpha}_{4-2m}} \\
&\leqslant C\epsilon^{1-2m}\Vert
u_2\Vert_{C^{4,\alpha}_{3-2m}}\epsilon^{-1}\Vert
u_2\Vert_{C^{4,\alpha}_{3-2m}} = O(\epsilon^\kappa),
\end{aligned}\]
using \eqref{eq:u2h2} again. 
\end{proof}

\subsection{Solving the non-linear equation}\label{sec:nonlinear}
We are finally ready to
try solving the equation we need to, i.e. we want $u,h$ such that
\[ \mathbf{s}(\Omega_2 + \ddbar u) = \l_{\Omega_2 + \ddbar
  u}(\mathbf{s} + h_1 + h_2 + h). \]
Expanding this in terms of the linearized operator, we have
\[ \begin{aligned}
\mathbf{s}(\Omega_2) + L_{\Omega_2}(u) + Q_{\Omega_2}(u) &=
\l_{\Omega_2}(\mathbf{s} + h_1 + h_2) 
+\frac{1}{2}\nabla u\cdot\nabla\l(\mathbf{s} + h_1 + h_2) \\
&\quad +
\l_{\Omega_2}(h) + \frac{1}{2}\nabla u\cdot\nabla h,
\end{aligned}\]
which we can write as
\[ \widetilde{G}_1(u,h) = \l_{\Omega_2}(\mathbf{s}+h_1+h_2) -
\mathbf{s}(\Omega_2) + \frac{1}{2}\nabla u\cdot\nabla h -
Q_{\Omega_2}(u), \]
where $\widetilde{G}_1$ is defined by 
\[ \widetilde{G}_1(u,h) := L_{\Omega_2}(u) - \frac{1}{2}\nabla u\cdot
\nabla\l(\mathbf{s} + h_1 + h_2) - \l_{\Omega_2}(h). \]

It follows from Remark~\ref{rem:G}, and Lemma~\ref{lem:Omega2} that
this operator $\widetilde{G}_1$ is sufficiently close to the operator
$G_1$ in Proposition~\ref{prop:inverse1} when $\epsilon \ll 1$, so
that the inverse $P_1$ from
Proposition~\ref{prop:inverse1} can be used to obtain an inverse
$\widetilde{P}_1$ for $\widetilde{G}_1$ with uniformly bounded
norm. We are therefore trying to solve the fixed point problem
\[ (u,h) = \mathcal{N}(u,h), \]
where 
\[ \begin{aligned}
      \mathcal{N}(u,h) : C^{4,\alpha}_\delta \times
      \overline{\mathfrak{h}} &\to C^{4,\alpha}_\delta \times
      \overline{\mathfrak{h}} \\
        (u,h) &\mapsto  \widetilde{P}_1\Big(\l_{\Omega_2}(\mathbf{s} + h_1 +
h_2) - \mathbf{s}(\Omega_2) + \frac{1}{2}\nabla u \cdot\nabla h -
Q_{\Omega_2}(u)\Big),
\end{aligned} \]
and $\delta = 4-2m + \tau$ for sufficiently small $\tau > 0$. 

The following lemma is essentially identical to Lemma 23 in
\cite{GSz10}. 
\begin{lem}\label{lem:contract}
There is a constant $c_1 > 0$ such that if 
\[ \Vert v_i\Vert_{C^{4,\alpha}_2}, |g_i| < c_1, \]
for $i=1,2$, then 
\[ \Vert \mathcal{N}(v_1,g_1) - \mathcal{N}(v_2,g_2)
\Vert_{C^{4,\alpha}_\delta \times \overline{\mathfrak{h}}} \leqslant
\frac{1}{2} \Vert (v_1-v_2, g_1-g_2)\Vert_{C^{4,\alpha}_\delta\times
  \overline{\mathfrak{h}}}. \]
\end{lem}

We can now complete the proof of Theorem~\ref{thm:gluing}. 
\begin{proof}[Proof of Theorem~\ref{thm:gluing}]
From Lemma~\ref{lem:Omega2} we have
\[ \Vert \mathcal{N}(0,0)
\Vert_{C^{4,\alpha}_\delta\times\overline{\mathfrak{h}}} 
\leqslant c_2\epsilon^{\kappa'}, \]
for some $\kappa' > 2m$. Define
\[ S = \{ (v,g)\,:\, \Vert v\Vert_{C^{4,\alpha}_\delta}, |g| \leqslant
2c_2\epsilon^{\kappa'}\}. \]
If $(v,g)\in S$, then for sufficiently small $\epsilon$ we have $|g| <
c_1$ with the $c_1$ from Lemma~\ref{lem:contract} and also 
\[ \Vert v\Vert_{C^{4,\alpha}_2} \leqslant C\epsilon^{\delta -2}\Vert
v\Vert_{C^{4,\alpha}_\delta} \leqslant 2Cc_2\epsilon^{\kappa' +
  \delta -2} < c_1, \]
for sufficiently small $\epsilon$, since $\kappa' + \delta -2 > 0$. It
follows that 
\[ \Vert \mathcal{N}(v,g)\Vert \leqslant \Vert \mathcal{N}(v,g) -
\mathcal{N}(0,0)\Vert + \Vert\mathcal{N}(0,0)\Vert \leqslant
\frac{1}{2}\Vert (v,g)\Vert + c_2\epsilon^{\kappa'} \leqslant
2c_2\epsilon^{\kappa'}, \]
so $\mathcal{N}$ is a contraction mapping $S$ into itself. We can
therefore find a fixed point $(u,h)$ of $\mathcal{N}$ in $S$, and this
gives a solution of the equation
\[ \mathbf{s}(\Omega_2 + \ddbar u) = \l_{\Omega_2+\ddbar u}
(\mathbf{s} + h_1 + h_2 + h). \]
From Lemma~\ref{lem:Gamma}, Equation~\eqref{eq:u2h2} and the fact that
$|h| \leqslant 2c_2\epsilon^{\kappa'}$, we have
\[ \begin{aligned}
  \mathbf{s} + h_1 + h_2 + h &= \mathbf{s} + h_1 + O(\epsilon^\kappa)
\\
&= \mathbf{s} - \epsilon^{2m-2}\frac{2\pi}{(m-2)!}(V^{-1}+\mu(p)) +
\epsilon^{2m}\frac{\mathbf{s}(p)}{m!}(V^{-1} + \mu(p)) \\
&\qquad - d_1\epsilon^{2m}\frac{2\pi^m}{(m-2)!}\Delta\mu(p) +
O(\epsilon^\kappa),
\end{aligned}\]
for some $\kappa > 2m$. This is what we wanted to prove. Since the
solution is obtained using the contraction mapping principle, the
solution will depend smoothly on the parameters. In particular when we
perform this construction at the set of all $T$-invariant points in
$M$, then the constant in $O(\epsilon^\kappa)$ can be chosen to be
uniform.  
\end{proof}

\section{Relative stability of blowups}\label{sec:relstab}
In this section we will give the proof of
Theorem~\ref{thm:Kstab}. Since we want to deal with K\"ahler manifolds
which are not necessarily algebraic, we will reformulate a simple
version of the usual
theory of K-stability in the K\"ahler setting, which is more similar
to Tian's original definition in \cite{Tian97} than to the more recent
algebro-geometric approach of Donaldson~\cite{Don02}.

\subsection{Relative K-stability for K\"ahler manifolds}
We will define relative K-stability for K\"ahler manifolds similarly
to Tian's definition~\cite{Tian97}. Our definition will actually be
simpler since we only consider test-configurations with smooth
central fibers.

\begin{defn}
A (smooth) test-configuration for a K\"ahler manifold $(M, \omega)$ consists of
a holomorphic submersion $\tau:\mathcal{X}\to\mathbf{C}$ together with
a holomorphic lift $v$ of the vector field
$\frac{\partial}{\partial\theta}$,
satisfying the following properties:
\begin{enumerate} 
  \item $\mathcal{X}$ admits a K\"ahler metric $\Omega$, for which
    $v$ is a Hamiltonian Killing field. 
  \item The pair $(\tau^{-1}(1), \Omega|_{\tau^{-1}(1)})$ is isometric
      to $(M, \omega)$.
\end{enumerate}
\end{defn}

The vector field $v$ gives a Hamiltonian Killing field on 
the central fiber $(M_0,\omega_0)$, which we will denote by $v$
also. Let us write $h_v$ for its Hamiltonian function. 
The Futaki invariant of this vector field is defined (see
Futaki~\cite{Fut83}, Calabi~\cite{Cal85}) to be
\begin{equation} \label{eq:Fut}
  \mathrm{Fut}(M_0, [\omega_0], v) = \int_{M_0}
  h_v(\overline{\mathbf{s}} - \mathbf{s}(\omega_0))\,\omega_0^n,
\end{equation}
where $\overline{\mathbf{s}}$ is the average of the scalar curvature
of $\omega_0$. 
The notation indicates that the Futaki
invariant does not depend on the particular metric chosen, only its
K\"ahler class.

For the definition of relative stability, we need to recall the
extremal vector field defined by Futaki-Mabuchi~\cite{FM95}. Let us
write $\mathrm{Aut}_0(M)$ for the connected component of the identity
in the automorphism group of $M$, and let
$\widetilde{\mathrm{Aut}}_0(M)$ be the kernel of the map from
$\mathrm{Aut}_0(M)$ to the Albanese torus of $M$. Finally, let
$G\subset\widetilde{\mathrm{Aut}}_0(M)$ be a maximal compact
subgroup. Suppose that $\omega$ is $G$-invariant. Then the action of
$G$ is Hamiltonian with respect to $\omega$ (see
e.g. LeBrun-Simanca~\cite{LS94}). This means that we can identify
elements
in the Lie algebra $\mathfrak{g}$ with their Hamiltonian functions,
normalized to have zero mean. Let us write
\[ \mathbf{s}_{ext} = \mathrm{proj}_{\mathfrak{g}}
\mathbf{s}(\omega), \]
for the $L^2$-projection of the scalar curvature of $\omega$ onto
$\mathfrak{g}$. The extremal vector field is the vector field
corresponding to $\mathbf{s}_{ext}$. The main result in \cite{FM95} is
that this vector field is independent of the choice of $G$-invariant
metric $\omega$. 

Suppose now that $T$ is a maximal torus in $G$. We say that the
test-configuration $(\mathcal{X},v)$ is compatible with $T$, if there is a
Hamiltonian holomorphic $T$-action on $\mathcal{X}$ preserving the fibers,
leaving $v$ invariant, and
such that when restricted to $\tau^{-1}(1)$ it recovers the $T$-action
on $M$. We define the modified Futaki invariant of such a
test-configuration as follows. As before, the central fiber of the
test-configuration has an induced Hamiltonian Killing field
$v$. Again, let $h_v$ be a Hamiltonian
function for $v$, normalized to have zero mean.
In addition, $(M_0,\omega_0)$ is
equipped with a Hamiltonian $T$-action, and in particular the extremal
vector field of $M$ induces a Hamiltonian holomorphic vector field
$v_{ext}$ on $M_0$, with a Hamiltonian function $h_{v_{ext}}$.
\begin{defn}
  The modified Futaki invariant of the test-configuration is defined
  to be
  \[ \mathrm{Fut}_{v_{ext}}(M_0,[\omega_0],v) = \int_{M_0}
  h_v(h_{v_{ext}} - \mathbf{s}(\omega_0))\,\omega_0^n. \]
\end{defn}

Note, in particular, that this coincides with the usual Futaki
invariant, if $h_v$ and $h_{v_{ext}}$ are orthogonal.

\begin{defn}
  We say that $(M,\omega)$ is K-semistable (with respect to
  smooth test-configurations), if
  \[ \mathrm{Fut}(M_0,[\omega_0],v)  \geqslant 0, \]
  for all smooth test-configurations, compatible with a maximal torus
  $T$ as above. If in addition equality only holds if the central
  fiber $M_0$ is biholomorphic to $M$, then $(M,\omega)$ is
  K-stable. Relative K-stability is defined analogously with the
  modified Futaki invariant replacing the Futaki invariant.  
\end{defn}

The following proposition follows from the theorem of Chen-Tian~\cite{CT05_1},
that the modified Mabuchi energy is bounded below, if $M$ admits an
extremal metric in the K\"ahler class $[\omega]$.
For test-configurations with smooth central fibers it is fairly
straight-forward to relate the modified Futaki invariant to the
behavior of the modified Mabuchi functional. This is
explained carefully in Tosatti~\cite{Tos12} and
Clarke-Tipler~\cite{CT13}.
As a consequence we have
the following. 

\begin{prop}\label{prop:Ksemistab}
If $M$ admits an extremal metric $\omega$, then $(M,\omega)$ is
relatively K-semistable (with respect to smooth test-configurations). 
\end{prop}

Using the method of \cite{Sto08} and
\cite{SSz09} we can improve the ``semistability'' to ``stability'',
but first we need to study Futaki invariants on blowups.

\subsection{Futaki invariants on blowups}\label{sec:FutBl}
Suppose that $(M,\omega)$ is a K\"ahler manifold, and $v$ is a
Hamiltonian holomorphic vector field on $M$, with Hamiltonian function
$h_v$. If $p\in M$ is such that $v$ vanishes at $p$, then $v$ lifts to
a holomorphic vector field $\hat{v}$ on the blowup $\Bl_p M$. Moreover, the lift
is Hamiltonian with suitable choices of K\"ahler metric on the blowup,
in the class $\pi^*[\omega] - \epsilon^2[E]$. We need to compute the
Futaki invariant
\[ \mathrm{Fut}(\Bl_pM, \pi^*[\omega]-\epsilon^2[E], \hat{v}) \]
in terms of the Futaki invariant on $M$. In the algebraic case this
computation was done by Stoppa~\cite{Sto10}, and was refined in
\cite{GSz10} (see also Della Vedova~\cite{DV08}).
On K\"ahler surfaces the first term of the expansion was
calculated by Li-Shi~\cite{LS12_1} under the assumption that we blow up
a non-degenerate zero of the vector field. We can obtain the general
result for K\"ahler manifolds using the
simple observation that the difference
\[ \mathrm{Fut}(\Bl_pM, \pi^*[\omega] - \epsilon^2[E], \hat{v}) -
\mathrm{Fut}(M, [\omega], v) \]
can essentially be computed in a neighborhood of $p$, since we can
choose the metric on the blowup to coincide with the metric $\omega$
outside a neighborhood of $p$. Indeed this is what the metric
$\omega_\epsilon$ in Section~\ref{sec:omegaepsilon} is like. We can
also choose the Hamiltonian function $h_{\hat{v}}$ to coincide with
$h_v$ outside a small ball $B$ around $p$, by choosing $h_{\hat{v}} =
\l(h_v)$. We then have the following formulas.
\begin{prop}\label{prop:Blformulas}
  For sufficiently small $\epsilon > 0$ we have
  \begin{equation}\label{eq:formulas}
    \begin{aligned}
         \int_{M}\omega^m - \int_{\Bl_p M} \omega_\epsilon^m &=
         \epsilon^{2m}, \\
         \int_M h_v\,\omega^m - \int_{\Bl_pM}
         \l(h_v)\,\omega_\epsilon^m &= \epsilon^{2m}h_v(p) +
         \frac{\epsilon^{2m+2}}{m+1}\Delta h_v(p), \\
         \int_M \mathbf{s}(\omega)\,\omega^m - \int_{\Bl_pM}
         \mathbf{s}(\omega_\epsilon) \omega_\epsilon^m &=  2\pi m(m-1)
           \epsilon^{2m-2}, \\
         \int_M h_v\mathbf{s}(\omega)\,\omega^m - \int_{\Bl_pM}
         \l(h_v)\mathbf{s}(\omega_\epsilon)\, \omega_\epsilon^m &=
         2\pi m(m-1)\epsilon^{2m-2} h_v(p) \\
         &\qquad +
         2\pi(m-2)\epsilon^{2m} \Delta h_v(p). 
  \end{aligned}
  \end{equation}
\end{prop}
\begin{proof}
  Each of these formulas can be reduced to a calculation on the
  projective space $\mathbf{P}^m$. The first and third formulas can
  also be checked easily using the cohomological interpretations of
  the integrals. The formulas involving $h_v$ could also be approached
  using equivariant cohomology, but we will not pursue this.

  In order to reduce the problem to a calculation on $\mathbf{P}^m$,
  note that each pair of integrals coincides outside a ball $B$ (in
  the notation of Section~\ref{sec:omegaepsilon} we are taking
  $B=B_{2r_\epsilon}$). So for instance we have
  \begin{equation}\label{eq:difference}
    \int_M \omega^m - \int_{\Bl_pM} \omega_\epsilon^m = \int_B
    \omega^m - \int_{\Bl_pB}\omega_\epsilon^m,
  \end{equation}
  with similar formulas for the other 3 integrals. 
  Since $v$ vanishes at $p$, we can choose
  coordinates around $p$ in which $v$ is given by a linear
  transformation. For small $\epsilon > 0$, we
  can therefore choose a metric $\Omega$ on
  $\mathbf{P}^m$ in the class $c_1(\mathcal{O}(1))$,
  together with a holomorphic Killing field $V$
  vanishing at a point $P\in \mathbf{P}^m$ (with Hamiltonian $h_V$),
  such that the data $(B_{2r_\epsilon}(P), \Omega, V, h_V)$ is
  equivalent in the obvious sense to
  the corresponding data $(B_{2r_\epsilon}(p), \omega, v,
  h_v)$. It follows from \eqref{eq:difference} together with the
  analogous formula on $\mathbf{P}^m$, that
  \[ \int_M \omega^m - \int_{\Bl_pM} \omega_\epsilon^m =
  \int_{\mathbf{P}^m} \Omega^m - \int_{\Bl_P\mathbf{P}^m}
  \Omega_\epsilon^m, \]
  where $\Omega_\epsilon$ is a metric on $\Bl_P\mathbf{P}^m$ in the
  class $\pi^*[\Omega] - \epsilon^2[E]$ constructed just like we
  constructed $\omega_\epsilon$. The analogous formula holds for all
  of the differences that we need to compute in
  Equation~\eqref{eq:formulas}. We have therefore reduced the problem
  to a calculation on $\mathbf{P}^m$. 

  On $\mathbf{P}^m$ one way to do the calculation would be to perform
  a computation in terms of toric geometry, since we can assume that
  $P$ is fixed by a maximal torus of automorphisms of
  $\mathbf{P}^m$. Alternatively, we can use an algebro-geometric
  calculation, since by continuity it is enough to deal with the case
  when  $\epsilon$ is rational. It is essentially this that we have
  already calculated in \cite[Lemma 28]{GSz10}. The results of that
  Lemma, together with the calculations in \cite[Proposition
  2.2.2]{Don02} imply the result we want. 
\end{proof}

Using this proposition we can compute the Futaki invariant on a
blowup, extending \cite[Corollary 29]{GSz10} to the K\"ahler case. 
\begin{cor}\label{cor:FutBl}
  Suppose that $h_v$ is normalized to have zero mean. 
  For sufficiently small $\epsilon > 0$ we have an expression
  \begin{equation}\label{eq:FB0}
     \mathrm{Fut}(\Bl_pM, [\omega_\epsilon], \hat{v}) = \mathrm{Fut}(M,
     [\omega], v) + A_\epsilon h_v(p) + B_\epsilon \Delta h_v(p), 
   \end{equation}
  where $A_\epsilon = O(\epsilon^{2m-2})$ and
  $B_\epsilon=O(\epsilon^{2m})$ 
  are functions of $\epsilon$
  depending on $(M, [\omega])$. One can easily expand $A_\epsilon$,
  $B_\epsilon$ in terms of $\epsilon$. In general 
  \begin{equation}\label{eq:FB1}
    \mathrm{Fut}(\Bl_pM, [\omega_\epsilon], \hat{v}) = \mathrm{Fut}(M,
  [\omega], v) + 2\pi m(m-1)\epsilon^{2m-2} h_v(p) +
  O(\epsilon^{2m}).
  \end{equation}
  Suppose that $\mathrm{Fut}(M, [\omega], v) = 0$, and $m > 2$. Then
  \begin{equation} \label{eq:FutBl} \begin{aligned}
       \mathrm{Fut}(\Bl_pM, [\omega_\epsilon], \hat{v}) &=
       2\pi m(m-1)\epsilon^{2m-2}h_v(p) \\
       &\qquad +
      \epsilon^{2m}\Big( 2\pi (m-2)\Delta h_v(p) - \overline{\mathbf{s}}
      h_v(p)\Big) + O(\epsilon^{2m+2}),
  \end{aligned} \end{equation}
  where $\overline{\mathbf{s}}$ is the average scalar curvature of $(M,
  \omega)$. 
  In addition if $h_v(p) = \Delta h_v(p) =0$, then $\mathrm{Fut}(\Bl_pM,
  [\omega_\epsilon], \hat{v})=0$. 
\end{cor}
\begin{proof}

  Let us write $\overline{\mathbf{s}}_\epsilon$ for the average scalar
  curvature of $\omega_\epsilon$. Then from \eqref{eq:formulas} we
  have
  \begin{equation}\label{eq:sepsilon}
    \begin{aligned}
    \overline{\mathbf{s}}_\epsilon &= \frac{\int_{\Bl_pM}
    \mathbf{s}(\omega_\epsilon)\, \omega_\epsilon^m}{\int_{\Bl_pM}
    \omega_\epsilon^m} = \frac{ \int_M \mathbf{s}(\omega)\,\omega^m -
    2\pi m(m-1)\epsilon^{2m-2}}{\int_M \omega^m - \epsilon^{2m}}  \\
   &= \overline{\mathbf{s}} + O(\epsilon^{2m-2}),
  \end{aligned} 
  \end{equation}
  and
  \begin{equation}\label{eq:f1} \begin{aligned}
    \int_{\Bl_pM} \l(h_v)\,\omega_\epsilon^m &= -\epsilon^{2m}h_v(p) -
    \frac{\epsilon^{2m+2}}{m+1} \Delta h_v(p), \\
    \int_{\Bl_pM}
    \l(h_v)\mathbf{s}(\omega_\epsilon)\,\omega_\epsilon^m &= \int_M
    h_v\mathbf{s}(\omega)\,\omega^m -2\pi
    m(m-1)\epsilon^{2m-2} h_v(p) \\ &\qquad - 2\pi (m-2)\epsilon^{2m}\Delta
    h_v(p). 
  \end{aligned}
\end{equation}
Combining these, and using that
\[ \mathrm{Fut}(M, [\omega],v) = -\int
h_v\mathbf{s}(\omega)\,\omega^m \]
since $h_v$ has integral zero, we get
 \[ \begin{aligned}
    \mathrm{Fut}(\Bl_pM, [\omega_\epsilon], \hat{v}) &= \int_{\Bl_pM}
    \l(h_v) (\overline{\mathbf{s}}_\epsilon -
    \mathbf{s}(\omega_\epsilon))\, \omega_\epsilon^m \\
    &= \overline{\mathbf{s}}_\epsilon \int_{\Bl_pM} \l(h_v)\,\omega_\epsilon^m
    - \int_{\Bl_pM}
    \l(h_v)\mathbf{s}(\omega_\epsilon)\,\omega_\epsilon^m  \\
    &= \overline{\mathbf{s}}_\epsilon\left(-\epsilon^{2m}h_v(p) -
      \frac{\epsilon^{2m+2}}{m+1} \Delta h_v(p)\right) - \int_M
    h_v\mathbf{s}(\omega)\, \omega^m \\
    &\qquad + 2\pi m(m-1)\epsilon^{2m-2} h_v(p) +
    2\pi(m-2)\epsilon^{2m} \Delta h_v(p) \\
    &= \mathrm{Fut}(M,[\omega],v) + A_\epsilon h_v(p) + B_\epsilon
    \Delta h_v(p). 
  \end{aligned}\]
  Here
  \[\begin{gathered} A_\epsilon = 2\pi m(m-1)\epsilon^{2m-2} -
    \epsilon^{2m}\overline{\mathbf{s}}_\epsilon, \\
    B_\epsilon = 2\pi (m-2)\epsilon^{2m} -
    \frac{\epsilon^{2m+2}}{m+1}\overline{\mathbf{s}}_\epsilon .
    \end{gathered}\]
  Using this together with the formula \eqref{eq:sepsilon}, we can
  obtain all the results that we are trying to prove.
\end{proof}

We need one more result, relating the inner product of holomorphic
Killing fields, introduced by
Futaki-Mabuchi~\cite{FM95}. This inner product is simply the $L^2$
product of the Hamiltonian functions, which are normalized to have
zero mean. So if $v, w$ have Hamiltonians $h_v, h_w$, normalized to
have zero mean on $M$, then
\[ \langle v,w\rangle = \int_M h_v h_w\,\omega^m. \]
If $v,w$ vanish at $p\in M$, then the product of the lifts to
$\Bl_pM$ is given by
\begin{equation}\label{eq:product}
  \langle \hat{v}, \hat{w}\rangle = \int_{\Bl_pM}
  \l(h_v)\l(h_w)\,\omega_\epsilon^m - \frac{1}{V_\epsilon} \int_{\Bl_pM}
\l(h_v)\,\omega_\epsilon^m \int_{\Bl_pM}
\l(h_w)\,\omega_\epsilon^m, 
\end{equation}
where $V_\epsilon$ is the volume of $\Bl_pM$ with respect to
$\omega_\epsilon$. The crucial property of this inner product is that
it is independent of the representative $\omega_\epsilon$ of its
K\"ahler class.

\begin{prop}\label{prop:prodBl} 
  Assume that $\langle v,w\rangle = 0$. Then for any $\delta > 0$ we
  have on the blowup $\Bl_pM$, that
  \[ \langle \hat{v},\hat{w}\rangle = O(\epsilon^{2m-\delta}).\]
  If in addition $h_v$ is normalized to have zero mean on $M$, and
  $h_v(p)=0$, then for any $\delta > 0$ we have
  \[ \langle\hat{v},\hat{w}\rangle = O(\epsilon^{2m+2-\delta}). \]
  In fact we could even take $\delta =0$ in both formulas, but we will
  not need this.
\end{prop}
\begin{proof}
We can assume that $h_v$ and $h_w$ are normalized to have zero mean on
$M$. Then from \eqref{eq:formulas} we know that the averages of
$\l(h_v)$ and $\l(h_w)$ on $\Bl_pM$ are of order $\epsilon^{2m}$, so
in the formula \eqref{eq:product} for $\langle\hat{v},\hat{w}\rangle$
we can ignore  the integrals of $\l(h_v)$ and $\l(h_w)$. 

Since $\l(h_v)\l(h_w) = h_vh_w$ and $\omega_\epsilon=\omega$ outside
$B_{2r_\epsilon}$, we just need to estimate the integrals on
$B_{2r_\epsilon}$ with respect to the different metrics $\omega$ and
$\omega_\epsilon$. 
The volume of $B_{2r_\epsilon}$ is
$O(r_\epsilon^{2m})$ with respect to both $\omega$ and
$\omega_\epsilon$ so we obtain
\[ \langle \hat{v},\hat{w}\rangle = O(r_\epsilon^{2m}). \]
If in addition $h_v(q)=0$, then we have $h_v\in C^0_2$, since also
$\nabla h_v(q)=0$ by our assumption. It follows that also $\l(h_v)\in
C^0_2$. This implies that
\[ \int_{B_{2r_\epsilon}} h_vh_w\,\omega^m \leqslant
C\int_0^{2r_\epsilon} r^2r^{2m-1}\,dr = O(r_\epsilon^{2m+2}), \]
and also
\[ \int_{B_{2r_\epsilon}} h_vh_w\,\omega_\epsilon^m \leqslant C\left(
  \int_\epsilon^{2r_\epsilon} r^2r^{2m-1}\,dr +
  \epsilon^2\epsilon^{2m}\right) = O(r_\epsilon^{2m+2}). \] 
We can do the construction with $r_\epsilon = \epsilon^\alpha$ for any
$\alpha < 1$, and choosing $\alpha$ sufficiently close to 1 we obtain
the results we want. Note that one can do the analogous calculation
algebro-geometrically and get a more precise result like in
Proposition~\ref{prop:Blformulas}, but we will not need this. 
\end{proof}

We can now improve Proposition~\ref{prop:Ksemistab} following
Stoppa~\cite{Sto08} and \cite{SSz09} to get
\begin{prop}\label{prop:Kstab}
 Suppose that $M$ admits an extremal metric $\omega$. Then
 $(M,[\omega])$ is relatively K-stable (with respect to smooth
 test-configurations).  
\end{prop}

The proof is essentially identical to the argument in \cite{SSz09},
using Proposition~\ref{prop:Ksemistab} together with the formulas that
we have shown in this section. In fact our situation is simpler since we are
only dealing with test-configurations with smooth central fiber.

As an application we have the
following proposition, which shows that we can only hope to construct
extremal metrics on blowups $\Bl_pM$, for which
$\nabla\mathbf{s}(\omega)$ vanishes at $p$. 
\begin{prop}
Suppose that $(M,\omega)$ is an extremal K\"ahler manifold, and $p\in
M$ is such that $\nabla\mathbf{s}(\omega)$ does not vanish at
$p$. Then $(\Bl_pM, [\omega_\epsilon])$ is relatively K-unstable for
all sufficiently small $\epsilon > 0$. 
\end{prop}
\begin{proof}
Let $T\subset G_p$ be a maximal
torus, with Lie algebra $\mathfrak{t}$. By our assumption,
$\mathbf{s}\not\in\mathfrak{t}$. We have an orthogonal decomposition
\[ \mathbf{s} = \mathbf{s}^\perp + \mathbf{s}_{\mathfrak{t}}, \]
where $\mathbf{s}_{\mathfrak{t}}\in\mathfrak{t}$ and $\mathbf{s}^\perp
\perp\mathfrak{t}$. Then $\nabla\mathbf{s}^\perp(p)\not=0$, and we can
assume that $\nabla\mathbf{s}^\perp$ generates a
$\mathbf{C}^*$-action. If it did not, we could approximate
$\mathbf{s}^\perp$ with elements of $\mathfrak{g}$ orthogonal to
$\mathfrak{t}$, which do generate $\mathbf{C}^*$-actions.  

Suppose then that $-\mathbf{s}^\perp$ generates the $\mathbf{C}^*$-action
$\lambda(t)$, and let $q = \lim_{t\to 0}\lambda(t)\cdot p$. This way
we obtain a test-configuration for $(\Bl_pM, \pi^*[\omega] - \epsilon^2[E])$ with central
fiber $\Bl_qM$. The Futaki invariant of the
test-configuration is given by
\[ \mathrm{Fut}(\Bl_qM, \pi^*[\omega] - \epsilon^2[E], -\hat{\mathbf{s}}^\perp) =
\mathrm{Fut}(M, [\omega], -\mathbf{s}^\perp) + O(\epsilon^{2m-2}), \]
using a calculation similar to Corollary~\ref{cor:FutBl}. Since 
\[ \mathrm{Fut}(M, [\omega], -\mathbf{s}^\perp) = \langle -\mathbf{s}^\perp,
\mathbf{s}\rangle = -\Vert \mathbf{s}^\perp\Vert^2, \]
we have a constant $c_0 > 0$ such that
\[  \mathrm{Fut}(\Bl_qM, \pi^*[\omega] - \epsilon^2[E], -\hat{\mathbf{s}}^\perp)   <
-c_0, \]
for sufficiently small $\epsilon$. In order to show that $(\Bl_pM,
L_\epsilon)$ is relatively K-unstable, we still need to adjust this
test-configuration to be orthogonal to $\mathfrak{t}$. Since
$\mathbf{s}^\perp$ is orthogonal to $\mathfrak{t}$, it follows from
Proposition~\ref{prop:prodBl}  that
for any $v\in\mathfrak{t}$ we have
\[ \langle \hat{\mathbf{s}}^\perp, v\rangle =
O(\epsilon^{2m-\delta}), \]
for any $\delta > 0$. This means that after modifying the
test-configuration with an element in $\mathfrak{t}$ to make it
orthogonal to $\mathfrak{t}$ on $\Bl_qM$, the Futaki invariant will
still be negative for sufficiently small $\epsilon$. It follows that
$(\Bl_pM, L_\epsilon)$ is relatively K-unstable for sufficiently small
$\epsilon > 0$. 
\end{proof}

\subsection{Test-configurations for blowups}
In this section we will give the proof of Theorem~\ref{thm:Kstab}. 
Suppose that $(M,\omega)$
is cscK , and suppose that $v$ is a
Hamiltonian holomorphic vector field on $M$ generating a
$\mathbf{C}^*$-action $\lambda(t)$. Then for any $p\in M$,
$\lambda(t)$ induces a test-configuration for
$(\Bl_pM,[\omega_\epsilon])$. The total space of this
test-configuration is simply the blowup of the product $M\times
\mathbf{C}$ along the closure of the orbit
\[ \{ (\lambda(t)\cdot p, t)\,:\, t\in\mathbf{C}^*\}. \]
If $q = \lim_{t\to 0}\lambda(t)\cdot p$, then the central fiber of the
test-configuration is $(\Bl_qM, \omega_\epsilon)$. The induced
$\mathbf{C}^*$-action on $\Bl_qM$ is given by $\lambda(t)$, which
lifts to $\Bl_qM$ since $q$ is a fixed point. The formula
\eqref{eq:FutBl} can be used to compute the Futaki invariant of this
test-configuration. 

When combined with Theorem~\ref{thm:main}, the following proposition
implies Theorem~\ref{thm:Kstab}. This proposition generalizes
\cite[Theorem 5]{GSz10} to K\"ahler manifolds.   
\begin{prop}\label{prop:unstable}
  Suppose that $n > 2$, and
  for some $\epsilon_0 > 0$ there does not
exist $q\in G^c\cdot p$ with $\mu(q) + \epsilon\Delta\mu(q)=0$
for any $\epsilon\in (0,\epsilon_0)$. Then
$(\Bl_pM,[\omega_\epsilon])$ is K-unstable for all sufficiently
small $\epsilon > 0$. 
\end{prop}
\begin{proof}
By moving $p$ in its $G^c$-orbit, we can assume that $G_p$ is a
maximal compact subgroup of $G^c_p$. Then if $T\subset G_p$ is a
maximal torus, then $T^c\subset G_p^c$ is also a maximal torus. As in
Section~\ref{sec:finitedim} we will work with the group
$G_{T^\perp}$. The corresponding moment map $\mu_{T^\perp}$ is simply
the orthogonal projection of $\mu$ onto $\mathfrak{g}_{T^\perp}$. Our
assumption says that $p$ is unstable for the action of $G^c_{T^\perp}$
with respect to the moment map
\[ \mu_{T^\perp} + \epsilon\Delta\mu_{T^\perp} \]
for all sufficiently small $\epsilon > 0$. There are several cases to
consider separately.
\begin{itemize}
\item Suppose that $p$ is strictly unstable for the moment map
  $\mu_{T^\perp}$. This means that there is a $v\in\mathfrak{g}_{T^\perp}$
  generating a $\mathbf{C}^*$-action $\lambda(t)$, such that
   \[\lim_{t\to 0} \langle \mu(\lambda(t)\cdot p), v\rangle < 0. \]
   Write $q = \lim_{t\to 0}\lambda(t)\cdot p$. We then have $h_v(q) <
   0$. From Corollary~\ref{cor:FutBl} it follows that the
   corresponding test-configuration for $(\Bl_pM,[\omega_\epsilon])$ 
   has Futaki invariant
\begin{equation}\label{eq:F1}
 \mathrm{Fut}(\Bl_qM,[\omega_\epsilon],\hat{v}) < -c_0\epsilon^{2m-2}, 
\end{equation}
for some $c_0 > 0$.
This means that $(\Bl_pM,[\omega_\epsilon])$ is K-unstable.
\item Suppose that $p$ is semistable for the moment map
  $\mu_{T^\perp}$, and we can
  find a $v\in \mathfrak{g}_{T^\perp}$ generating a
  $\mathbf{C}^*$-action
  $\lambda(t)$, such that
\[ \begin{aligned} \lim_{t\to 0} \langle \mu(\lambda(t)\cdot p),
  v\rangle &=0 \\
\lim_{t\to 0} \langle \Delta\mu(\lambda(t)\cdot p), v\rangle < 0. 
\end{aligned}\]
Writing again $q = \lim_{t\to 0} \lambda(t)\cdot p$ we then have
$h_v(q) =0$ and $\Delta h_v(q) < 0$. From Corollary~\ref{cor:FutBl} we have
\[ \mathrm{Fut}(\Bl_qM, [\omega_\epsilon],\hat{v}) < -c_0\epsilon^{2m}, \]
for some $c_0 > 0$, and so it follows that $(\Bl_pM,
[\omega_\epsilon])$ is K-unstable.
\item In the remaining case $p$ is semistable with respect to
  $\mu_{T^\perp} + \epsilon\Delta\mu_{T^\perp}$ for all sufficiently small
  $\epsilon > 0$. This implies that we can find
  $q_\epsilon$ in the boundary $\partial G^c\cdot p$ of the
  $G^c$-orbit such that
\[ \mu_{T^\perp}(q_\epsilon) + \epsilon\Delta\mu_{T^\perp}(q_\epsilon)
=0. \]
Since $\partial G^c\cdot p$ is a finite union of orbits, at least one
orbit must contain infinitely many $q_\epsilon$. Choose
$q_{\epsilon_1}$ and $q_{\epsilon_2}$ to be in the same orbit. Since
the moment maps are equivariant and $T$ fixes $q_{\epsilon_i}$, we
have
\[ \mu(q_{\epsilon_i}) + \epsilon_i\Delta\mu(q_{\epsilon_i})\in
\mathfrak{t}. \]
In addition, the projection of the moment map to the stabilizer is an
invariant of the orbit, so if $q$ is 
 in the same $G^c_{T^\perp}$ orbit as the $q_{\epsilon_i}$, then
\[ \mathrm{pr}_{\mathfrak{g}_q}
(\mu(q) + \epsilon_i\Delta\mu(q))\in
\mathfrak{t}. \]
Since this holds for at least two different $\epsilon_i$, we must have
\begin{equation}\label{eq:prmu}
\mathrm{pr}_{\mathfrak{g}_q}\mu(q),
\mathrm{pr}_{\mathfrak{g}_q}\Delta\mu(q) \in\mathfrak{t}. 
\end{equation}
It follows that 
 the stabilizer of $q_{\epsilon_1}$ in $G^c_{T^\perp}$ is
reductive and so there is a local slice for the action of
$G^c_{T^\perp}$ near $q_{\epsilon_1}$ (see Sjamaar~\cite{Sja95} or 
Snow~\cite{Snow82}). Using the
Hilbert-Mumford criterion applied to
the action of the stabilizer on the tangent space at $q_{\epsilon_1}$, 
we can find  a $v\in\mathfrak{g}_{T^\perp}$ generating
a $\mathbf{C}^*$-action $\lambda(t)$, and a point $p'\in G^c\cdot p$, 
such that $q_{\epsilon_1} = \lim_{t\to
  0}\lambda(t)\cdot p'$. This means that there exists a
test-configuration for $(\Bl_pM,[\omega_\epsilon])$, whose central fiber is
$(\Bl_qM,[\omega_\epsilon])$, writing $q=q_{\epsilon_1}$. 

We claim that this test-configuration has zero Futaki
invariant. Indeed, since $v\in \mathfrak{t}^\perp$, it follows from
\eqref{eq:prmu} that the Hamiltonian $h_v$ satisfies $h_v(q) = \Delta
h_v(q) = 0$. In addition since $\mathbf{s}\in\mathfrak{t}$,
we have $\mathrm{Fut}(M,[\omega],v) =
0$. Corollary~\ref{cor:FutBl} then implies that
$\mathrm{Fut}(\Bl_qM,[\omega_\epsilon],\hat{v}) = 0$. It
follows that $(\Bl_pM,[\omega_\epsilon])$ is K-unstable. 
\end{itemize}
\end{proof}

Combining our results we can prove Theorem~\ref{thm:Kstab}.
\begin{proof}[Proof of Theorem~\ref{thm:Kstab}]
$(1) \Rightarrow (2)$: This follows from
Proposition~\ref{prop:Kstab}.

$(2) \Rightarrow (3)$: This is the statement of
Proposition~\ref{prop:unstable}.

$(3) \Rightarrow (1)$: It follows from Theorem~\ref{thm:main},
that under the assumption the blowup $\Bl_pM$ admits an extremal
metric in the class $[\omega_\epsilon]$. We just need to check that
this metric has constant scalar curvature. To do this we need to
compute the Futaki invariant $\mathrm{Fut}(\Bl_pM, [\omega_\epsilon],
\hat{v})$ for all $v\in\mathfrak{g}_p$. Since $(M,\omega)$ is cscK we know that
$\mathrm{Fut}(M,[\omega],v) =0$. In addition if $\epsilon$ is
sufficiently small, then an argument similar to the proof of
Proposition~\ref{prop:alldelta} shows that 
\[ \mathrm{pr}_{\mathfrak{g}_p} \mu(p), \mathrm{pr}_{\mathfrak{g}_p}
\Delta\mu(p) =0. \]
Therefore $h_v(p) = \Delta h_v(p) =0$, so from
Corollary~\ref{cor:FutBl} we get $\mathrm{Fut}(\Bl_pM,
[\omega_\epsilon], \hat{v}) =0$. It follows that the extremal metric
constructed using Theorem~\ref{thm:main} has constant scalar
curvature. 
\end{proof}

\bibliographystyle{plain}
\bibliography{../../mybib}

\end{document}